\documentclass[12pt]{amsart}
\usepackage{amssymb,epsfig,amsmath,latexsym,amsthm}
\usepackage{graphicx}
\theoremstyle{plain}
\newtheorem{theorem}{Theorem}[section]
\newtheorem{lemma}[theorem]{Lemma}

\newtheorem{proposition}[theorem]{Proposition}
\newtheorem{corollary}[theorem]{Corollary}
\newtheorem{question}[theorem]{Question}
\newtheorem{questions}[theorem]{Questions}

\newtheorem{problem}[theorem]{Problem}
\theoremstyle{definition}
\newtheorem{definition}[theorem]{Definition}
\newtheorem{definition/construction}[theorem]{Definition/Construction}
\newtheorem{construction}[theorem]{Construction}
\newtheorem{remark}[theorem]{Remark}

\newtheorem{acknowledgements}[theorem]{Acknowledgements}

\newtheorem{remarks}[theorem]{Remarks}

\newtheorem{notation}[theorem]{Notation}

\DeclareMathOperator{\Diff}{Diff}

\DeclareMathOperator{\Emb}{Emb}
\DeclareMathOperator{\Maps}{Maps}

\DeclareMathOperator{\id}{id}

\DeclareMathOperator{\rel}{rel}

\DeclareMathOperator{\inte}{int}

\DeclareMathOperator{\fix}{fix}

\newcommand{\piinv}{\pi^{-1}}

\newcommand{\BN}{\mathbb N}

\newcommand{\BR}{\mathbb R}
\newcommand{\BZ}{\mathbb Z}

\newcommand{\mA}{\mathcal{A}}

\newcommand{\mD}{\mathcal{D}}

\newcommand{\mF}{\mathcal{F}}

\usepackage{epsfig,color}

\begin{document}

\title{Self-Referential Discs and the Light Bulb Lemma}

\author{David Gabai}\address{Department of Mathematics\\Princeton
University\\Princeton, NJ 08544}

\thanks{Version 0.47, November 8, 2020.  Partially supported by NSF grants DMS-1607374 and DMS-2003892}
 
 \email{gabai@math.princeton.edu}

\begin{abstract} We show how self-referential discs in 4-manifolds lead to the construction of pairs of discs  with a common geometrically dual sphere which are  homotopic rel $\partial$, concordant and coincide near their boundaries, yet are not properly isotopic.  This occurs in manifolds without 2-torsion in their fundamental group, e.g.  the boundary connect sum of $S^2\times D^2$ and $S^1\times B^3$, thereby exhibiting phenomena not seen with spheres.  On the other hand we show that two such discs are isotopic rel $\partial$ if the manifold is simply connected.  We construct in $S^2\times D^2\natural S^1\times B^3$ a properly embedded 3-ball properly homotopic to a $z_0\times B^3$ but not properly isotopic to $z_0\times B^3$.  \end{abstract}

\maketitle

\setcounter{section}{-1}

\section{Introduction}\label{S0}

 In its simplest form the \emph{light bulb lemma} \cite{Ga} asserts that if a surface $R$ in the 4-manifold $M$ has a geometrically dual sphere $G$, then one can perform the \emph{crossing change} of Figure 1 (Figure 2.1 in \cite{Ga}) via an isotopy of $R$, provided there is a path $\alpha$ from $y$ to $z=R\cap G$ that is disjoint from the tube $B$.  Recall that a geometrically dual sphere is an embedded sphere $G$ with trivial normal bundle that intersects $R$ once and transversely.   This paper investigates what happens when such path $\alpha$ must cross $B$, i.e. is \emph{self-referential}.  It leads to the discovery of homotopic, concordant but non isotopic discs with common geometrically dual spheres, thereby exhibiting new phenomena not seen for spheres in a large class of manifolds.  It also leads to the discovery of knotted 3-balls in certain 4-manifolds.
 
\begin{figure}[ht]
\includegraphics[scale=0.80]{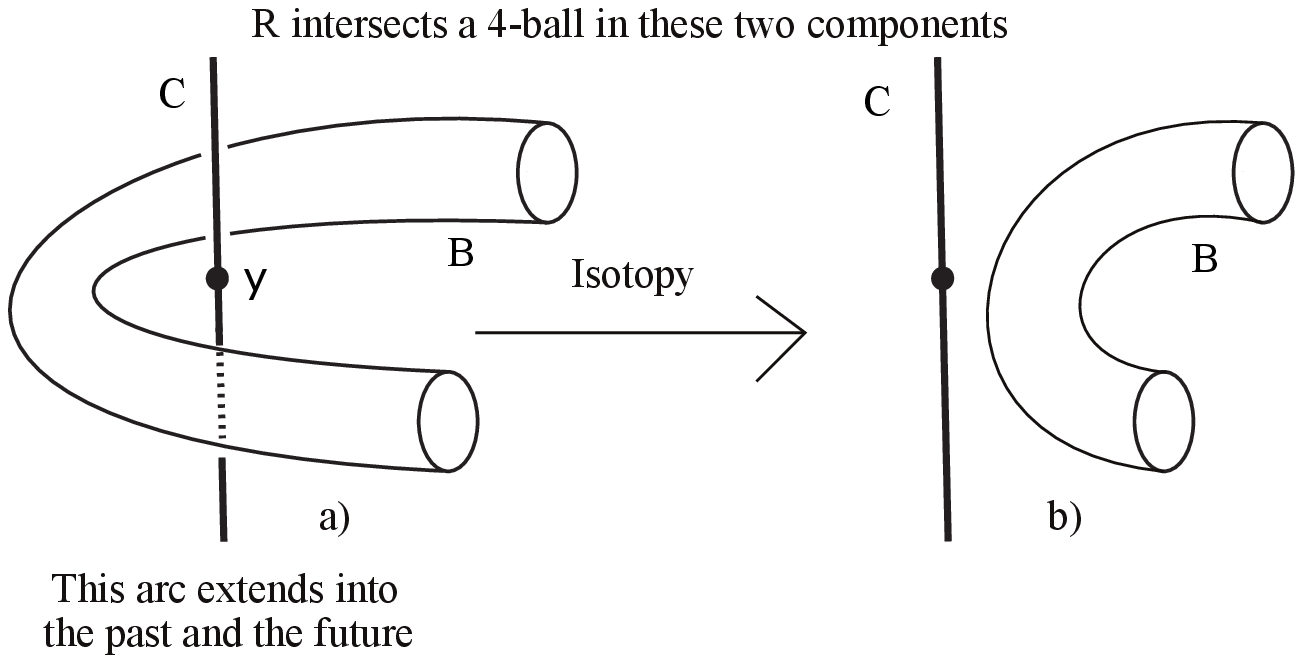}
\caption{}
\end{figure}

Perhaps the simplest example is shown in Figure 2.  Here $V=S^2\times D^2\natural S^1\times B^3 := W\times [-1,1]$ where $W$ is a solid torus with an open 3-ball removed.  Let  $G$ denote the 2-sphere component of $\partial W_0$, where $W_0=W\times 0$.  Let $D_0$ be a vertical disc in the $S^2 \times D^2$ factor and $P$ a round 2-sphere centered in $W_0$ that projects to a disc in $W_0$ disjoint from $D_0$.   See Figure 2 a).  Note that $D_0\cap W_0$ (resp. $P \cap W_0$) is an arc (resp. a circle).  Let $D_1$  be obtained by tubing the disc $D_0$ to the 2-sphere $P$, such that the projection of $D_1$ to $W_0$ is as in Figure 2 b).  Here $D_1\cap W_0$ is an arc and the shading indicates projections from the past and future to $W_0$.  Note that $D_0$ and $D_1$ have the common geometrically dual sphere $G$.  If we could apply the light bulb lemma to $D_1$ near where the tube links the sphere, then $D_1$ is isotopic to $D_0 \rel \partial$.   

Here is the idea for showing that $D_0$ and $D_1$ are non isotopic $\rel \partial$.  Let $I_0$ denote the arc $D_0\cap W_0$ oriented to point into $G$ and $\Emb(I,V; I_0)$ the space of proper arc embeddings based at $I_0$ that coincide with $I_0$ near $\partial I_0$.  Then $D_0, D_1$ naturally correspond to loops $\alpha_0, \alpha_1$ in $\Emb(I,V; I_0)$ where $\alpha_0$ is the constant loop.  Using methods from Dax \cite{Da} we will show that $\alpha_1$ is not homotopic to $\alpha_0$ in $\Emb(I,V; I_0)$ and hence $D_1$ is  not isotopic to $D_0 \rel \partial$.

\setlength{\tabcolsep}{60pt}
\begin{figure}
 \centering
\begin{tabular}{ c c }
 $\includegraphics[width=3.5in]{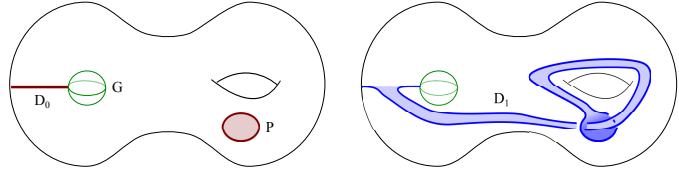}$  
\end{tabular}
 \caption[(a) X; (b) Y]{\label{Figure2}
 \begin{tabular}[t]{ @{} r @{\ } l @{}}
A Self-Referential Disc
\end{tabular}}
\end{figure}

\begin{remarks} \label{failure remarks} i) Let $M$ be a 4-manifold such that $\pi_1(M)$ has no 2-torsion.  Theorem 1.2 \cite{Ga} shows that if two homotopic 2-spheres $A_0, A_1\subset M$ have a common geometrically dual sphere $G$ and coincide near $G$, then they are ambiently isotopic fixing a neighborhood of $G$ pointwise.  Since the isotopy is supported in a disc in the domain, I initially thought that Theorem 1.2 proved that properly homotopic discs with geometrically dual spheres are properly isotopic.  However, the proof of Theorem 1.2 uses that $A_0$ is a sphere as opposed to a disc in one crucial spot.  See  Remark \ref{key point}.  

ii) On the other hand, there is nothing new when $G\subset S^2\times S^1\subset \partial M$, for filling this component with a $S^2 \times D^2$ reduces to the study of isotopy classes of spheres with geometrically dual spheres. That was solved for spheres in 4-manifolds $M$ such that $\pi_1(M)$ has no 2-torsion in \cite{Ga} and in general 4-manifolds by Schneiderman and Teichner \cite{ST}.

iii) Hannah Schwartz \cite{Sch} showed that there exist manifolds with 2-torsion in their fundamental groups supporting homotopic spheres with a common geometric dual that are not isotopic, in fact not even concordant.  Rob Schneiderman and Peter Teichner \cite{ST} identified the Freedman - Quinn (FQ) concordance invariant \cite{FQ} as the exact obstruction and showed that concordance implies isotopy.  

iv) Note that $D_1$ is concordant to $D_0$, thus their difference is not detected by the FQ invariant.   A secondary obstruction to isotoping one sphere to another is the \emph{km} invariant of Stong \cite{St} which is only defined when FQ=0.  See \cite{KM} for a modern exposition.  The Stong invariant does not detect that $D_1$ is not isotopic to $D_0$.  First, one can attempt to transform the isotopy problem for discs to one for spheres by attaching a 0-framed 2-handle to $V$ along $\partial D_0$ and extending $D_0$ and $D_1$ to spheres, but then these spheres become isotopic by \cite{Ga}. Secondly, $km=0$ when the spheres have a common geometrically dual sphere.     \end{remarks}

We now define our obstruction generally and introduce the work of Dax before stating our main results.
\vskip 8pt

\begin{construction} (An obstruction to isotopy) \label{obstruction} Let  $D_0$ be a properly embedded disc in the 4-manifold $M$.  View $D_0$ as $I\times I$ with $I_0$ denoting $I\times 1/2$ and $ \mF_0$  this product foliation. If $D$ is another properly embedded disc that coincides with $D_0$ near $\partial D_0$, then $D$ gives rise to a canonical  element $[\phi_{D_0}(D)]\in \pi_1(\Emb(I,M;I_0))$, where $\Emb(I,M;I_0)$ is the space of smooth embeddings of $I$ based at $I_0$.     To see this, view $ D=I\times I$ where this foliation $\mF$ coincides with $\mF_0$ near  $\partial D_0$.   Use $D_0$ to inform how to modify $\mF$ to a loop $\phi_{D_0}(D)$ in $\Emb(I,M;I_0)$ based at $ I_0$.  (See Definition \ref{novel} for more details.)  Since $[\phi_{D_0}(D_0)] = [1_{I_0}]$, where $1_{I_0}$ is the constant map to $I_0$, and $\Diff(D^2\fix \partial)$ is connected \cite{Sm3}, the class $[\phi_{D_0}(D)] \in \pi_1(\Emb(I,M;I_0))$ is well defined and gives an obstruction to isotoping $D$ to $D_0 \rel \partial D_0$. \end{construction}

Let $f_0: N^n\to M^m$ be an embedding where $N$ and $M$ are closed manifolds.  In 1972 Jean-Pierre Dax showed \cite{Da} that $\pi_k(\Maps(N,M), \Emb(N,M), f_0)$ is isomorphic to a certain bordism group when $2\le k\le 2m-3n-3$.  While stated very abstractly,  the case $N=I$ and $M$ a 4-manifold can be restated with a strikingly elegant formulation.  This paper gives that reformulation a self contained exposition.  See \S 3.  Let $\pi_1^D(\Emb(I,M; I_0))$ denote the subgroup of $ \pi_1(\Emb(I,M; I_0))$ represented by loops that are inessential in $\Maps(I,M:I_0)$. The following result is a slightly stronger version of the restated Theorem A \cite{Da} p.345 for $N=I$ and $M$ a 4-manifold.

\begin{theorem} (Dax Isomorphism Theorem) \label{dax} Let $I_0$ be an oriented properly embedded $[0,1]$ in the oriented 4-manifold $M$.  Then 

i) There is a homomorphism $d_3:\pi_3(M, x_0)\to \BZ[\pi_1(M)\setminus 1]$ with image $D(I_0)$, called the \emph{Dax kernal}.

ii) $\pi_1^D(Emb(I,M;I_0))$ is  canonically isomorphic to $\BZ[\pi_1(M)\setminus 1]/D(I_0)$ and generated by $\{\tau_g|g\neq 1, g\in \pi_1(M)\} $.  \end{theorem}

\begin{remark} The $\tau_g$'s arise from a spinning construction of Ryan Budney.  See Definition \ref{spin}.\end{remark}

Thus Construction \ref{obstruction} together with the Dax isomorphism theorem gives a concrete obstruction to isotoping  one embedded disc to another $\rel \partial$.

\begin{corollary} Let $D_0$ be a  properly embedded disc in the oriented 4-manifold and $\mD$ be the isotopy classes of embedded discs homotopic $\rel \partial$ to $D_0$, then there is a canonical function $\phi_{D_0}: \mD \to \BZ[\pi_1(M)\setminus 1]/D(I_0)$ such that if $D$ is a embedded disc homotopic rel $\partial$ to $D_0$, then $\phi_{D_0}([D])\neq 0$ implies $D$ is not isotopic to $D_0\rel\partial$.\end{corollary}

Note that $\phi_{D_0}$ is a function of $D_0$.
\vskip 8pt

In the setting of properly embedded discs with a common dual sphere, the methods of \cite{Ga} show that $\phi_{D_0}$ is a homomorphism whose image contains a particular subgroup and also proves the converse when $\pi_1(M)=1$.

\begin{theorem}  \label{main} Let $M$ be a compact 4-manifold and $D_0$ a properly embedded 2-disc with a geometrically dual sphere $G\subset \partial M$.  Let $\mD$ be the isotopy classes of embedded discs homotopic $\rel \partial$ to $D_0$. 

i)  If $\pi_1(M)=1$, then $\mD=[D_0]$, i.e. if $D_0$ and $D_1$ are homotopic rel $\partial$, then they are isotopic $\rel \partial$.  

ii) In general, $\mD$ is an abelian group with zero element $[D_0]$.  There is a homomorphism $\phi_{D_0}:\mD\to   \BZ[\pi_1(M)\setminus 1]/D(I_0)\cong \pi_1^D(\Emb(I,M; I_0))$.   It maps onto the subgroup generated by elements of the form $g+g^{-1}$ and $\hat \lambda$, where $\hat\lambda^2=1$.  \end{theorem}

\begin{remarks} i)  We shall see in \S 4 that for $M=S^2\times D^2\natural S^1\times B^3$ the Dax kernal is trivial and the disc $D_1$ of Figure 2 maps to $t+t^{-1}$, thus $D_0$ and $D_1$ are not isotopic rel $\partial$.

ii) $\mD$ is a torsor when there is a dual sphere.  Fixing the element $[D_0]$ turns it into a group with identity $[D_0]$. $ \BZ[\pi_1(M)\setminus 1]$ acts on $\mD$ by adding self-referential tubes and $\BZ[T_2]$ acts on $\mD$ by adding double tubes, where $T_2$ is the set of non trivial 2-torsion elements.  See \S 4.  \end{remarks}


As an application we show the existence of knotted 3-balls in 4-manifolds.

\begin{theorem}\label{knotted ball}  If $V=S^2\times D^2\natural S^1\times B^3$ and $B_0= x_0\times B^3$, then there exists a properly embedded 3-ball $B_1\subset V$ such that $B_1$ is properly homotopic but not properly isotopic to $B_0$.  See Figure 3.\end{theorem}

Here is the idea of the proof.  An extension of Hannah Schwartz' Lemma 2.3 \cite{Sch} to discs implies that there is a diffeomorphism $\phi:V\to V$ fixing a neighborhood of $\partial V$ pointwise  and  homotopic to $\id$ rel $\partial$ such that $\phi(D_0)=D_1$.  Let $B_0$ denote the 3-ball $x_0\times B^3$ in the $S^1\times B^3$ factor of $V$ and $B_1:=\phi(B_0)$.  If $B_1$ is isotopic to $B_0$, then since $B_1$ is disjoint from $D_1$, $D_1$ can be isotoped into the $S^2\times D^2$ factor of $V$. Theorem 10.4 \cite{Ga} implies that $D_1$ is isotopic to $D_0$ rel $\partial$, a contradiction. $B_1$ is obtained from $B_0$ by embedded surgery as described in more detail in \S 5.  See Figure 3.     \vskip 10 pt

\setlength{\tabcolsep}{40pt}
\begin{figure}
 \centering
 \begin{tabular}{ c c }
 $\includegraphics[width=5in]{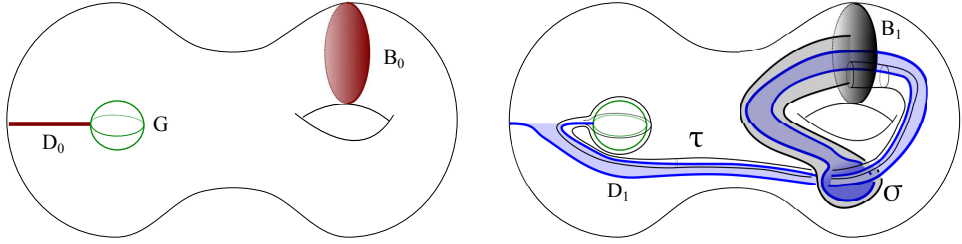}$  \end{tabular}
 \caption[(a) X; (b) Y]{\label{fig:Fig1}\begin{tabular}[t]{ @{} r @{\ } l @{}}
 A Knotted 3-Ball
 \end{tabular}}
\end{figure}


This paper is organized as follows.  Basic definitions will be given in \S1.  Section \S2 will describe to what extent the methods of \cite{Ga} extend to discs.  In particular we will show that if $D_0$ and $D_1$ are homotopic and have a common dual sphere, then $D_1$ can be put into a \emph{self-referential form} with respect to $D_0$.  This is the analogue of the normal form of \cite{Ga} except that in addition to double tubes, $D_1$ can have finitely many self-referential discs.   Theorem \ref{main} i) will also be proved.  The Dax isomorphism theorem \cite{Da} will be stated and proved in \S3.  
A slightly sharper version of Theorem \ref{main} ii) will be proved in \S4.  Applications to knotted 3-balls in 4-manifolds and further questions will be given in \S5.  

\begin{acknowledgements}  We thank Hannah Schwartz for helpful conversations and Ryan Budney for his comments and for teaching me about the modern theory of embedding spaces.
We thank the Max Planck Institute in Bonn  and the Banff International Research Station for the opportunity to present the main results at workshops respectively in September and November 2019.  Much of this paper was written while a member of the Institute for Advanced Study.  \end{acknowledgements}

\section{Basic Definitions}

We say that $G$ is a \emph{dual sphere} for the properly embedded disc $D\subset M$ if $G\subset \partial M$ and $D$ intersects $G$ exactly once and transversely.  It would be more proper to call such a $G$ a \emph{geometrically dual boundary} sphere to distinguish it from geometrically dual spheres intersecting $D$ at an interior point. A \emph{geometric dual sphere} is one with trivial normal bundle that intersects a given surface exactly once and transversely.  Trivial normal bundle is automatic here since $G$ is an embedded homologically non trivial sphere in an orientable 3-manifold.  Unless said otherwise all dual spheres for discs lie in the boundary of the 4-manifold.

If $S_0$ and $S_1$ are oriented surfaces, then we say that they are tubed \emph{coherently} if the tubing creates an oriented surface whose orientation agrees with that of $S_0$ and $S_1$.

This paper works in the smooth category.  All manifolds are orientable.

\section{Self-Referential Form}

Let $D_0$ be a properly embedded disc with dual sphere $G\subset \partial M$.  In this section we show that if $D_1$ is an embedded disc with $\partial D_0=\partial D_1$  and $D_1$ is homotopic $\rel \partial$ to $D_0$, then $D_1$ can be isotoped to a \emph{self-referential form}, i.e. $D_1$ looks like $D_0$ except for finitely many double tubes representing distinct non trivial 2-torsion elements of $\pi_1(M)$ and self-referential discs.
 
\setlength{\tabcolsep}{60pt}
\begin{figure}
 \centering
\begin{tabular}{ c c }
 $\includegraphics[width=3.5in]{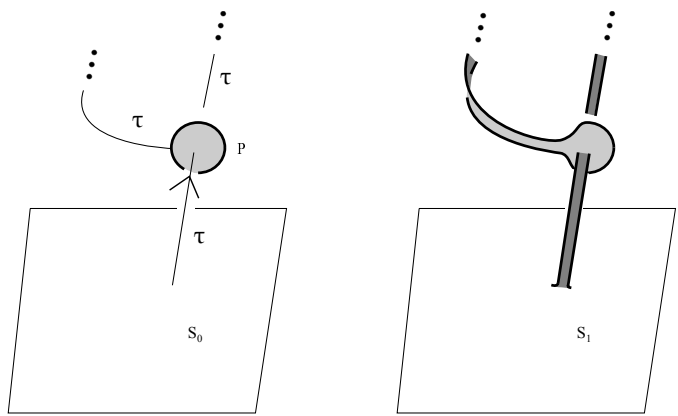}$  
\end{tabular}
 \caption[(a) X; (b) Y]{\label{FigureB,1}
 \begin{tabular}[t]{ @{} r @{\ } l @{}}
A Self-Referential Disc
\end{tabular}}
\end{figure}

\begin{definition}  Let $S_0$ be a properly embedded oriented surface in the 4-manifold $M, B\subset \inte(M)$ an oriented embedded 3-ball with $B\cap S_0=\emptyset$ and $\partial B=P$.  Let $\tau:[0,1]\to M$ be an embedded path from $\inte(S_0)$ to $P$ such that $\tau(0)=\tau\cap S_0, \tau(1)=\tau \cap P$ and $\inte(\tau)$ intersects $B$ exactly once and transversely.  Let $S_1$ be obtained from $S_0$ by tubing $S_0$ to $P$ along $\tau$.  We say that $S_1$ is obtained from $S_0$ by attaching a \emph{self-referential disc}.  See Figure \ref{FigureB,1}.\end{definition}

\begin{remarks}  i)  The disc $D_1$ in Figure 2 is obtained by attaching a self-referential disc to the disc $D_0$. 

ii) A priori to define the tubing, $\tau$ should be a framed embedded path as in Definition 5.4 \cite{Ga}.  Up to isotopy supported in $N(\tau)$ there are four isotopy classes, exactly two of which are coherent with the orientations of $S_0$ and $P$.  These two, as do the non coherent ones, differ by the non trivial element of $\pi_1(SO(3))$ on the $B^3$ normal fibers of $N(\tau)$ as one traverses $\tau$.  Since $\tau$ attaches to a sphere, the two choices give isotopic $S_1$'s.  Thus $S_1$ depends only on $\tau$ and coherence/noncoherence. Equivalently, we can fix the orientation of the sphere one way or the other and then insist that the attachment be coherent. \end{remarks}

\begin{definition} Now assume that $D_0\subset M$ is a properly embedded oriented disc with dual sphere $G$.  Let $B\subset\inte(M)$ an oriented 3-ball with $\partial B=P$ and $B\cap D_0=\emptyset$.  Let $\tau_0$ be an embedded arc from $\inte(D_0)$ to $\inte(B)$ intersecting $B\cup D_0$ only at its endpoints.  Think of it as being very short and view $D_0\cup \tau_0\cup B$ as  the basepoint for $\pi_1(M)$.  Associated to $g\in \pi_1(M)$ and $\sigma\in \pm$ construct $D_1$ by attaching a self-referential disc as follows.  Let $\tau_1 $ be a path from $B$ to $\inte(D_0)\setminus \tau_0$ such that $\tau_1(0)=\tau_0(1), \tau_1\cap(D_0\cup\tau_0\cup B)=\partial \tau_1$ and $\tau_1$ represents the class $g$.  Use $\tau=\tau_0 * \tau_1$ to construct $ D_1$ where $\sigma$ determines whether or not the attachment is coherent.  See Figure 4.

Given $\sigma_1 g_1, \cdots, \sigma_n g_n$ construct a disc $D_1$ by attaching $n$ self-referential discs to $D_0$ by starting with $n$ adjacent copies of $\tau_0 \cup B$ and then attaching $n$ self-referential discs as above.  \end{definition}

\begin{remark}   Since $D_0$ has a dual sphere the inclusion $M\setminus(D_0\cup \tau_0\cup B)\to M$ induces a $\pi_1$-isomorphism.  Thus once $B$ and $\tau_0$ are chosen, if $D_1$ is obtained by attaching one self-referential disc, then  $D_1$ is determined up to isotopy by $\sigma$ and $g$.   In a similar manner, if $D_1$ is obtained by attaching n self-referential discs, then once the $n$ adjacent copies of $\tau_0 \cup B$ are chosen it is determined up to isotopy by $\sigma_1 g_1, \cdots, \sigma_n g_n$.   \end{remark}

The statement of \emph{self-referential form} given in Defintion \ref{sr form} below is  quite technical, so for now we give the following informal one. Starting with $D_0$ construct the normal form analogue of Definition 5.23 and Figure 5.10 \cite{Ga} and then attach self-referential discs to obtain $D_1$.  The actual definition includes some constraints and keeps track of certain orientations. The following is the main result of this section.

\begin{theorem}  \label{srf} Let $D_0, D_1$ be properly embedded discs in the 4-manifold $M$ that coincide near their boundaries and have a geometrically dual sphere $G\subset \partial M$.  If $D_0$ and $D_1$ are homotopic $\rel \partial$, then $D_1$ can be isotoped rel $\partial$ to self-referential form with respect to $D_0$.\end{theorem}

Before embarking on the proof we recall the following result which is a rewording of Theorems 1.2 and 1.3 \cite{Ga}.

\begin{theorem}  \label{light bulb} Let $M$ be a 4-manifold such that the embedded spheres $R_0$ and $R_1$ have a  common geometrically dual sphere $G$ and coincide near $G$.  If $R_1$ and $R_0$ are homotopic and $\pi_1(M)$ has no 2-torsion, then they are ambiently isotopic fixing $N(G)$ pointwise.  In general $R_1$ can be ambiently isotoped fixing $N(G)$ pointwise to be in normal form with respect to $R_0$. \end{theorem}

\begin{remarks} \label{key point} i) As mentioned in the introduction, since the isotopy fixes $N(G)$ pointwise, I originally thought that this theorem is a result about properly homotopic discs with dual spheres, which seems to contradict the main result of this paper.  


ii) The key point is this.  In the proof of Theorem \ref{light bulb} the dual sphere is repeatedly used to enable various geometric operations.  When $R_1$ is a sphere, $\partial N(G)=S^2\times S^1$.  Therefore, if $z=R_1\cap G$, then through each point of $\partial N(z)\cap R_1$ there is a distinct dual sphere.  On the other hand, when $D_1$ is a disc we assume that $G\subset \partial M$ and so $N(G)=G\times I$.  Here there may only be an interval $[a,b]\subset \partial D_1$ with the property that for $\theta\in [a,b]$, $D_1$ has a distinct dual sphere through $\theta$.  For example, consider the disc $D_1$ of Figure 2.  For most of the proof of Theorem \ref{light bulb} an interval suffices, but near the end, at one crucial spot, we require the whole circle.  See the second paragraph preceding Lemma 8.1 \cite{Ga} where it is stated ``We can further assume that $q_1\in\partial D_0$." Note that when $G\subset S^2\times S^1\subset \partial M$, each point of $\partial D_0$ sees its own dual sphere, so the proofs of \cite{Ga} and \cite{ST} apply to discs without modification.

iii) There is the temptation to push $G$ to $G'\subset\inte(M)$ and use $G'$ as a dual sphere; however, an argument along the lines of  \cite{Ga} requires that $D_1$ be $G'$-inessential, a condition automatic for spheres but not for discs.\end{remarks}

\begin{definition}  Parametrize $\partial D_0=\partial D_1$ by $[0,2\pi]/{\sim}$ and $N(G)\cap \partial M$ as $G\times [\pi/2,3\pi/2]$ so that $\partial D_0\cap (G\times \theta)= \theta$.  Call $[\pi/2,3\pi/2]\subset \partial D_0$ the \emph{approach interval}.  \end{definition}
The proof of Theorem \ref{light bulb} extends essentially directly to the proof of Theorem \ref{srf} until the third paragraph of \S 8. We now elaborate on this extension and then state a result that summarizes what survives for discs.  
\vskip 8pt 
\noindent\emph{Section 2}: The extension is direct.  In particular, the light bulb lemma goes through unchanged.  
\vskip 8pt
\noindent\emph{Section 3}: Not relevant.\vskip 8pt
\noindent\emph{Section 4}: Smale's theorem implies that embedded discs  that are homotopic $\rel \partial$ are properly regularly homotopic $\rel \partial$.  \vskip 8pt

\setlength{\tabcolsep}{60pt}
\begin{figure}
 \centering
\begin{tabular}{ c c }
 $\includegraphics[width=1in]{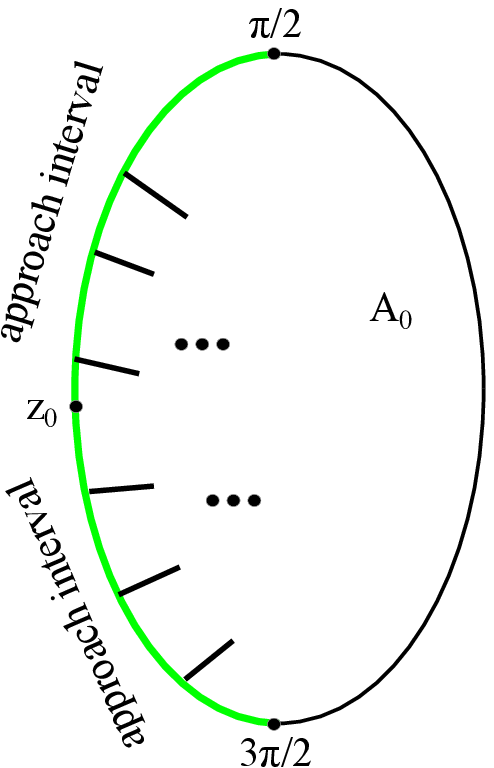}$  
\end{tabular}
 \caption[(a) X; (b) Y]{\label{FigureB,2}
 \begin{tabular}[t]{ @{} r @{\ } l @{}}
Tubed Surface
\end{tabular}}
\end{figure}

\noindent\emph{Section 5}: 1) Definition of \emph{tubed surface}.  Recall that a tubed surface $\mA$ is the data for constructing an embedded surface in $M$.  In the end  of the proof of our Theorem \ref{srf} above the associated surface $A_1$ will be our $D_0$ and the realization $A$ will be our $D_1$.  While stated for closed surfaces, the  definition of a tubed surface applies to compact surfaces with boundary.    For us, $A_0$ is a disc with $\partial A_0$ parametrized by $[0,2\pi]/{\sim}$ where $[\pi/2,3\pi/2]$ is the approach interval, $z_0=\pi\in \partial A_0$ and $f(z_0)=z=A_1\cap G$.  In the closed surface setting we can assume that the $\sigma, \alpha, \beta, \gamma$ tube guide curves approach $z_0\in A_0$ radially.  In the disc setting these curves approach $[\pi/2,3\pi/2]\subset\partial A_0$ transversely and intersect $N(\partial A_0)$ in distinct arcs.  See Figure \ref{FigureB,2}.  That figure shows $\partial A_0$ together with the tube guide curves in a small neighborhood of the approach interval, which is shown in green.

2) Construction of the realization $ A$.  The construction is essentially the same.  Here a tube guide curve $\kappa$ connecting to $\theta\in \partial A_0$ corresponds to a tube paralleling $f(\kappa)\subset A_1$ that connects  to a parallel copy of $G\times \theta$ pushed slightly into $\inte(M)$.  

3) Tube sliding moves.  With one exception all the moves yield isotopic realizations as before.  In the disc setting, the \emph{reordering move} between tube guide curves $\kappa_j, \kappa_k$ requires that the relevant component between their endpoints lies in the approach interval.  

4) Finger and tube locus free Whitney moves.  Same as before.

5) Theorem 5.21.  The proof is the same as before, in particular reordering is not used.

6) Lemma 5.25.  The proof holds since one can permute pairs $(\beta_i, \gamma_i), (\beta_j, \gamma_j)$ that are adjacent in the approach interval.  

Summary:  Except for a restricted reordering move, all the results of Section 5 directly hold.
\vskip 8pt

\noindent\emph{Section 6}:  Direct analogues of all the results of this section hold for discs.  Here are some additional remarks.  

1) Lemma 6.1 holds tautologically since $D_0$ and $D_1$ are homtopic $\rel \partial$.

\begin{notation}\label{sign convention} Sign Convention:  We continue to adopt the orientation convention on $\beta_i, \lambda_i$ and $\gamma_i$ as in that section.  As in 6.3 \cite{Ga} the tube guide curve $\alpha$ corresponds to a sphere $P(\alpha)$ obtained by connecting oppositely oriented copies of $G$ by a tube that parallels $f(\alpha)$.  Orient $\alpha $ so that the copy giving $-[G]$ (resp. $[G])$ is at the negative (resp. positive) end of $f(\alpha)$.  \end{notation}

2) If $\pi:\tilde M\to M$ is the universal covering map, then the components of $\pi^{-1}(D_1\cup G)$ are in natural 1-1 correspondence with elements of $\pi_1(M, z)$ and the components of $\pi^{-1}(G)$ freely generate a $\BZ[\pi_1(M)]$ submodule of $H_2(\tilde M)$, thus the algebra of \S 6 extends to the disc case.

3) In our context the associated surface $A_1$ in the statement of Proposition 6.9 is a disc.  The proof is a direct translation.\vskip8pt

\noindent\emph{Section 7:}  The statement and proof of the crossing change lemma hold as before.\vskip 8pt

\noindent\emph{Section 8}:  The proof holds as before, until the second to last sentence of the third paragraph.  That sentence ``We can further assume that $q_1\in \partial D_0$."  requires that the approach interval is the whole circle.  \vskip 8pt

\setlength{\tabcolsep}{60pt}
\begin{figure}
 \centering
\begin{tabular}{ c c }
 $\includegraphics[width=1.5in]{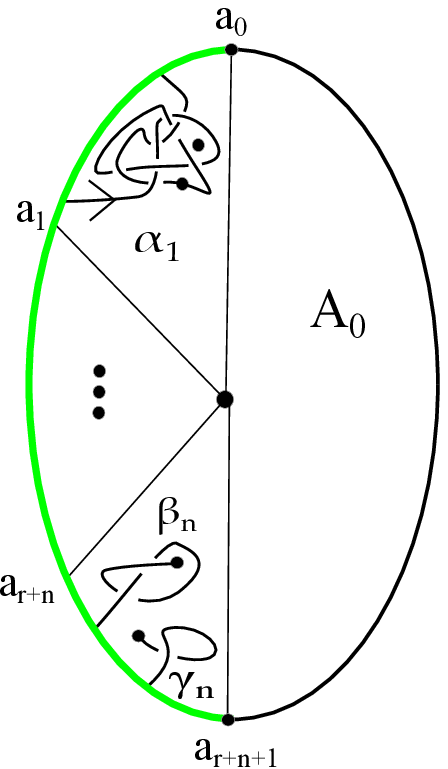}$  
\end{tabular}
 \caption[(a) X; (b) Y]{\label{FigureB,3}
 \begin{tabular}[t]{ @{} r @{\ } l @{}}
Sector Form
\end{tabular}}
\end{figure}

Putting this all together we have the following result.

\begin{proposition} (Sector Form)  \label{whats left} Let $D_0, D_1$ be properly embedded discs in the 4-manifold $M$ such that $D_0$ and $D_1$ coincide near their boundaries and have the dual sphere $G\subset \partial M$.  Then there exists a tubed surface $\mA$ with underlying surface $A_0$ parametrized as the unit disc in $\BR^2$, with $f(A_0)=D_0$ and with realization $A$ isotopic $\rel \partial$ to $D_1$.  
$\mA$ has data  $(\alpha_1,(p_1,q_1), \tau_1), \cdots, (\alpha_r,(p_r, q_r), \tau_r),(\beta_0, \gamma_0, \lambda_0), (\beta_1, \gamma_1,\lambda_1), \cdots, (\beta_n, \gamma_n, \lambda_n))$.

Each each of these data sets lie in distinct sectors of $A_0$.  This means that there exists linearly  ordered $a_0=\pi/2, a_1, \cdots, a_{r+n+1}=3\pi/2 \subset \partial A_0$ such that $(\alpha_i, (p_i,q_i))\subset$ the sector defined by $(a_{i-1},a_i, 0)$ and 
$(\beta_j,\gamma_j)$ lies in the sector defined by $(a_{r+j}, a_{r+j+1},0)$ with $\beta_j\cap \gamma_j=\emptyset$.  \emph{See Figure \ref{FigureB,3}}.  \end{proposition}

\begin{lemma}  \label{normal permutation} The data of the various sectors can be permuted without changing the isotopy class of the realization.  \end{lemma}

\begin{proof} Using the tube sliding operations any two adjacent pairs  $(\alpha_i,(p_i,q_i), \tau_i)$, $(\beta_j, \gamma_j, \lambda_j)$, i.e. two of one type or one of each type, in the approach interval  can be permuted, but we cannot \emph{permute} data within a given sector, i.e. the $\beta_i$ and $\gamma_i$ curves.  \end{proof} 

\begin{definition}  A tubed surface $\mA$ with data as in  Proposition \ref{whats left} is said to be in \emph{sector form}.  Let $\mA$ be a tubed surface in sector form.  Let $\lambda$ be a framed embedded path in $M$ with disjoint embedded tube guide curves $\beta$ and $\gamma \subset A_0$, all oriented with the above sign convention.  We denote the pair $(\beta,\gamma)$ as $+ (\beta,\gamma)$ (resp. $- (\beta,\gamma))$ if $\beta$ appears before (resp. after) $ \gamma$ in the approach interval.  Call an embedded $\alpha $ curve $+$ (resp. $-$) if the negative (resp. positive) end of $\alpha$ appears before the positive (resp. negative) end in the approach interval. \end{definition}

\begin{definition}\label{sr form}  We say that the tubed surface $\mA$ is in \emph{self-referential form} with data $(\lambda_1, \lambda_2, \cdots, \lambda_n, \sigma_1 g_1, \cdots, \sigma_k g_k)$ if

a) The immersion $f:A_0 \to M$ is a proper embedding with $f(A_0)=A_1$ a 2-disc with dual sphere $G\subset \partial M$.

b) The paths $\beta_1, \gamma_1, \cdots, \beta_n, \gamma_n, \sigma_1\alpha_1, \cdots, \sigma_k\alpha_k$ are embedded and linearly arrayed along the approach interval, where $\sigma_i\in \pm$ and $+\alpha_i$ (resp. $-\alpha_i$) denotes that its negative (resp. positive) end is closer to $\pi/2$ than its positive end.     The point $q_i$ associated to $\alpha_i$ lies in the half disc bounded by $\alpha_i$ and the approach interval. 

c) The framed embedded paths $\lambda_1, \lambda_2, \cdots, \lambda_n$ represent distinct nontrivial 2-torsion elements of $\pi_1(M)$.   

d) Each $g_i$ represents a non trivial element of $\pi_1(M, z_0)$ and no $i,j$ is $\sigma_i g_i=-\sigma_j g_j$.  

We say that the disc $D_1$ is in \emph{self-referential form} with data $(\lambda_1, \lambda_2, \cdots, \lambda_n, \sigma_1 g_1, \cdots, \sigma_k g_k)$ with respect to the disc $D_0$ if $D_1$ is the realization of the tubed surface $ \mA $ with this data where $A_1=D_0$.  \end{definition}

We now show the key connection between the formal definition and the earlier one for self-referential form.

\begin{lemma} \label{alpha to srf} If $D_1$ is in self-referential form with respect to $D_0$ with data $(\lambda_1, \lambda_2, \cdots, \lambda_n,\newline \sigma_1 g_1, \cdots, \sigma_k g_k)$ and $D'_0$ is in self-referential form with respect to $D_0$ with data $(\lambda_1, \lambda_2, \cdots,  \lambda_n)$, then $D_1$ is isotopic to the surface obtained from $D'_0$ by attaching the self-referential discs associated to the data $(\sigma_1 g_1, \cdots, \sigma_k g_k)$.\end{lemma}

\begin{proof} Since $q_1$ lies to the approach interval side of $\alpha_1$ sliding the sphere $P(\alpha_1)$ off of $D_0$ entangles the tube connecting $D_0$ to $P(\alpha_1)$ to create a self-referential disc of the type claimed.  See Figures \ref{Figure4,3} to \ref{Figure4,5}.  The result follows by induction on the number of $\alpha$ curves.\end{proof}


\setlength{\tabcolsep}{60pt}
\begin{figure}
 \centering
\begin{tabular}{ c c }
 $\includegraphics[width=5in]{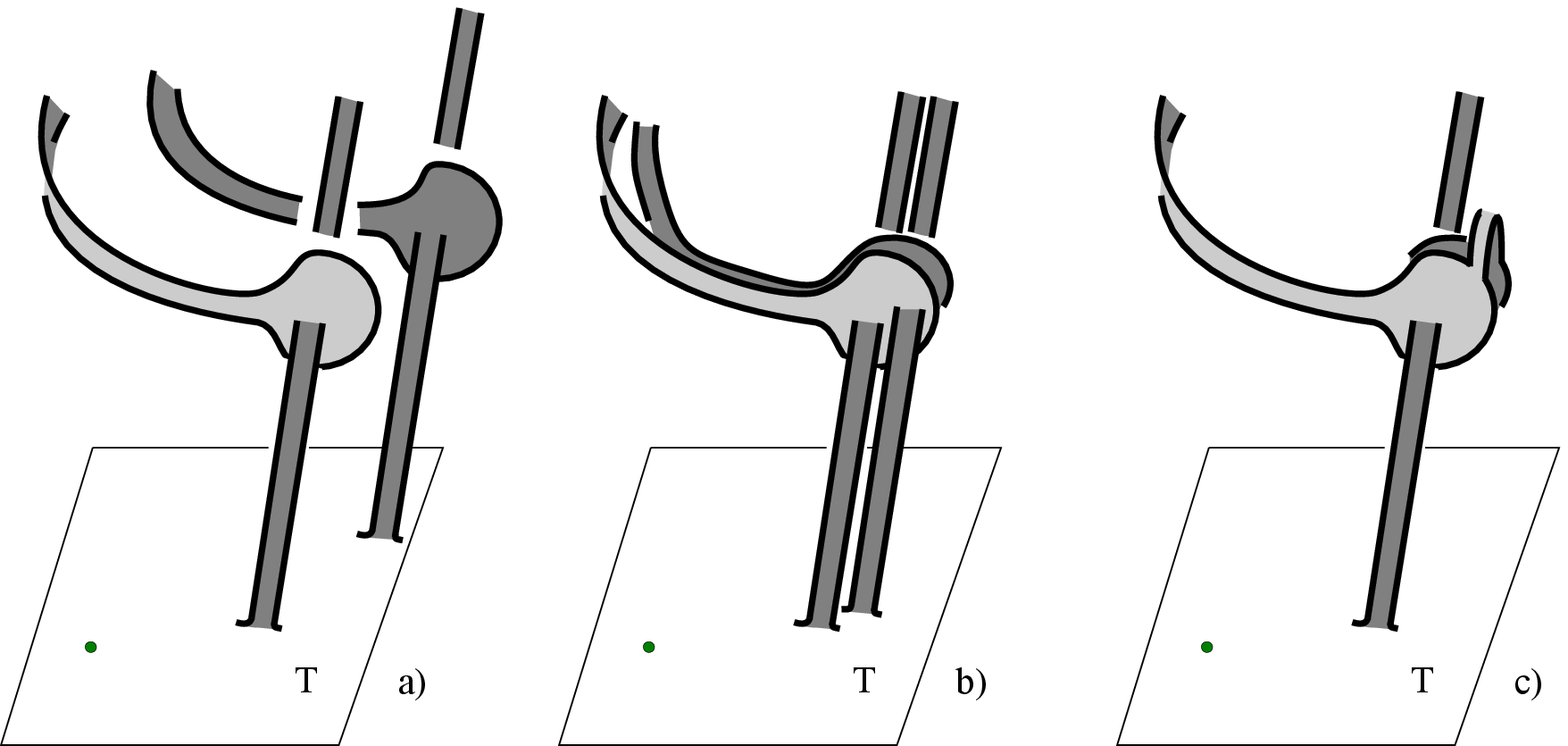}$  
\end{tabular}
 \caption[(a) X; (b) Y]{\label{FigureB,4}
 \begin{tabular}[t]{ @{} r @{\ } l @{}}
$D_g+D_{-g}=D_0$
\end{tabular}}
\end{figure}

\begin{lemma}\label{alpha cancellation} An embedded surface $T$ with dual sphere $G$ is isotopic to the surface $T'$ obtained from $T$ by tubing self-referential discs of type $g, -g$.\end{lemma}

\begin{proof} See Figure \ref{FigureB,4}.  Figure \ref{FigureB,4} a) shows $T$ with self-referential discs of type $g, -g$.  The green dot denotes intersection with a geometrically dual sphere, which is on $\partial T$, when $T$ is a disc.  Two applications of the light bulb lemma enable the isotopy to Figure \ref{FigureB,4} b).  Figure \ref{FigureB,4} c) is after sliding one of the tubes.  Since the spheres now cancel, that surface is isotopic to $T$ itself.\end{proof}

\begin{definition} We say that the embedded surface $T$ is obtained from the embedded surface $S$ by \emph{tubing a sphere $P$ along $\tau$}, if $P$ bounds a 3-ball disjoint from $S$ and $T$ is obtained by tubing $S$ and $P$ along a framed embedded path $\tau$.\end{definition}

\begin{lemma} \label{stabilization} Let $S$ be an embedded surface with dual sphere $G$.  If the surface $T$ is obtained from $S$ by tubing a sphere $P$ along $\tau$, then $T$ is isotopic to a surface obtained from $S$ by attaching finitely many self-referential discs.\end{lemma}

\begin{proof} If $P=\partial B$ and $|B\cap \tau|=k$, then squeeze $B$ into two balls $B_1, B_2$ so that $|\tau\cap B_1|=1, |\tau\cap B_2|=k-1$ and $(\partial \tau\cap B)\subset B_2\setminus B_1$.  If $P_i=\partial B_i$, then we can further assume that $P_1$ is connected to $P_2$ by a tube $ \tau_1$ disjoint from $\tau$.   Use $\tau$ to slide $\tau_1$ off of $P_2$ so that now $\tau_1$ connects $P_1$ with $S$.  Here we abused notation by identifying the framed embedded path $\tau$ with its corresponding tube.  By construction $\tau_1$ will link $P_1$ exactly once.  Next, use the light bulb lemma to unlink $\tau_2$ from $P_1$ and $\tau_1$ from  $P_2$.  The result follows by induction on $k$.  \end{proof}

\begin{lemma} \label{alpha stabilization}  Let $\mA$ be a tubed surface in sector form containing a sector $J$ with data $(\alpha_i,(p_i,q_i), \tau_i)$.  There exists another tubed surface $\mA'$ with isotopic realizations whose data agrees with that of $\mA$ except that the $(\alpha_i,(p_i,q_i), \tau_i)$ data has been deleted and the sector $J$ has been subdivided into finitely many sectors each of which contains data of the form $(\sigma_s\alpha_s,(p_s,q_s), \tau_s))$ where $\alpha_s$ is embedded and $q_s$ lies in the halfdisc bounded by $\alpha_s$ and the approach interval.\end{lemma}

\begin{proof} By the crossing change Lemma 7.1 \cite{Ga} we can assume that $\alpha_i$ is monotonically increasing.  Sliding $P(\alpha_i)$ off of $A_1$ as in the proof of Lemma \ref{alpha to srf} we obtain an unknotted 2-sphere $P_i$, which is entangled with $ \tau_i$.  If $S$ denotes the realization of the tubed surface $\mA$ with the data $(\alpha_i,(p_i,q_i), \tau_i)$ deleted, it follows that the realization $A$ of $\mA$ is obtained by tubing $S$ to the sphere $P_i$.  By Lemma \ref{stabilization} $A$ is isotopic to a surface obtained by adding self-referential discs to $S$.  The proof of that lemma further shows that they can be attached in subsectors of $J$ without the self-referential discs linking with other parts of $A$.  Finally, reverse the proof of Lemma \ref{alpha to srf} to obtain the desired $\mA'$ satisfying all but possibly the last conclusion.  If a $q_s$ lies outside the halfdisc bounded by $\alpha_s$ and the approach interval, then deleting the data $(\sigma_s\alpha_s,(p_s,q_s), \tau_s)$ does not change the isotopy class of the realization, \end{proof}

The next result follows from Lemmas \ref{alpha cancellation} and \ref{alpha stabilization}.

\begin{corollary}  \label{alpha addition} Let $\mA$ be a tubed surface in sector form.  Given  the data $(\alpha_s,(p_s,q_s), \tau_s)$ there exists a tubed surface $\mA'$ in sector form with realization isotopic to that of $\mA$ such that the data of $\mA'$ consists of the data from the sectors of $\mA$ plus another sector with data $(\alpha_s,(p_s,q_s), \tau_s)$ together with other sectors having data only involving $\alpha $ curves.\qed\end{corollary}

\noindent\emph{Proof of the Self-referential Form Theorem}.  By Proposition \ref{whats left} we can assume that  $\mA$ is in sector form.  

0) By Lemma \ref{normal permutation} the data of the various sectors can be  permuted.    

i) Elimination of the $(\beta_0, \gamma_0, \lambda_0)$ data  can be done as in Remark 8.2 \cite{Ga}.  This might create additional data of the form $(\alpha_s,(p_s,q_s), \tau_s)$.

ii) We can further assume that the $\lambda_i$'s represent distinct non trivial 2-torsion elements since the methods of \S 6 \cite{Ga} enable the exchange of a pair of double tubes representing the same 2-torsion element for a pair of single tubes.  Again, this might create data of the form $(\alpha_s,(p_s,q_s), \tau_s)$.

iii) The modification of the $\beta_i, \gamma_i$ curves to embedded tube guide curves can be done as in the two paragraphs after Remark 8.2 \cite{Ga}.  This might require that $\mA$ has particular sectors of the form $(\alpha_s,(p_s,q_s),\tau_s)$ in order to invert the operation of \S 6 \cite{Ga}.  We can create such sectors by Lemma \ref{alpha addition} at the cost of creating other sectors with data of the form $(\alpha_t,(p_t,q_t), \tau_t)$.  Also, the modification may create other sectors of this type.

iv) To reverse the ordering of the tube guide curves in $(\gamma_i,\beta_i, \lambda_i)$ where $\lambda_i$ represents 2-torsion, modify $\mA$ to create two new sectors with data of the form $(\beta_i, \gamma_i, \lambda_i), (\beta_i,\gamma_i, \lambda_i)$ at the cost of adding sectors with $(\alpha_s,(p_s,q_s), \tau_s)$ type data. Then cancel the $(\gamma_i,\beta_i, \lambda_i), (\beta_i, \gamma_i, \lambda_i)$ pairs at the possible cost of additional  type $(\alpha_s,(p_s,q_s), \tau_s)$ sectors.

v) Apply  Lemma \ref{alpha stabilization} to each sector with $(\alpha_s,(p_s,q_s), \tau_s)$ data.\qed.
\vskip 10 pt

If $\pi_1(M)=1$, then the self-referential form data is trivial, thus, we have proved the following, stated as Theorem \ref{main} i) in the introduction.

\begin{theorem} \label{pi trivial} Let $D_0, D_1$  be properly embedded discs in the 4-manifold that coincide near their boundaries and have the common dual sphere $G\subset \partial M$.  If $M$ is simply connected, then $D_1$ is homotopic to $D_0 \rel \partial $ if and only if it is isotopic  $\rel \partial$.  \end{theorem}


\section{The Dax Isomorphism Theorem}

Let $f_0: N^n\to M^m$ be an embedding where $N$ and $M$ are closed manifolds.  In 1972 J. P. Dax showed that $\pi_k(\Maps(N,M), \Emb(N,M), f_0)$ is isomorphic to a certain bordism group when $2\le k\le 2m-3n-3$. See Theorem A and Theorem 1.1 \cite{Da}.  While both the statement and proof are expressed in the very abstract and general style of the day, our case of interest is a strikingly clean and beautiful geometric result with an elementary proof.  Using different language and in part different methods we exposit this result when $N=I:=[0,1]$ and $f_0: I\to M^4$ is a proper embedding with image $I_0$.   Again, unless stated otherwise, all maps and spaces are smooth and in this section manifolds are oriented. Standard spaces are standardly oriented.

\begin{definition}  Define the \emph{Dax group} $\pi_1^D(\Emb(I,M; I_0))$ to be the subgroup of $\pi_1(\Emb(I,M; I_0))$ consisting of classes represented by loops in $\Emb(I,M;I_0)$ that are homotopically trivial in $\pi_1(\Maps(I,M; I_0))$.  Here $\Emb(I,M;I_0)$ (resp. $\Maps(I,M;I_0)$) is the based space of proper embeddings (resp. proper continuous maps) that coincide with $I_0$ near $\partial I_0$.  Here we abuse notation by identifying the interval $I_0$ with the embedding $f_0:I\to I_0$.  \end{definition}


The following definition is a special case of the \emph{spinning} operation of Ryan Budney \cite{Bu}, see Figure \ref{FigureD,1}.  That figure shows the projection of a 4-ball $B\subset M$ to a 3-ball $\hat B$.  Our path $\alpha_{t}$, which is constant near $t=.5$, intersects $B$ (resp. $\hat B$) in arcs $\sigma$ and $\tau$ (resp. $\sigma$ and a point).  It is modified to one where $\sigma$ spins about the point.  What follows is a slightly more formal definition.

\begin{definition}\label{spin} Let $\alpha_t:L\to M, t\in [0,1]$ be a path in $\Emb(L,M)$ where $L$ is an oriented 1-manifold and $M$ an oriented 4-manifold.  Assume that $\alpha_t$ is constant for $t\in [.45, .55]$.  Let $B\subset M$ be parametrized by $[-2,2]\times [-2,2]\times [-1,1]\times [-1,1]$.  With respect to local coordinates assume that $B\cap L=\sigma\cup\tau$ where $\tau = (0,0,0,-s), s\in [-1,1]$, $\sigma=\{-1,0,s,0\}, s\in [-1,1]$ and both are oriented from the $s=-1$ to the $s=+1$ end.  We modify $\alpha$ to $\gamma$ so that $\alpha_t(s)=\gamma_t(s)$ unless $t\in [.45,55]$ and $\alpha_{.5}(s)\in \sigma$.  Within $t\in [.45,55]$, keeping endpoints fixed and staying within the 2-sphere $Q\subset [-2,2]\times [-2,2]\times [-1,1]\times 0=\hat B$, swing  $\sigma$ around $\tau$ by  first going around the negative $y$-side and then back along the positive $y$-side of $Q$.  This can be done so that $\gamma_t$ is a smooth loop.  See Figure \ref{FigureD,1}.  We say that $\gamma$ is obtained by \emph{spinning} $\alpha$. Note that Lk($\tau$,Q)=+1, where (motion of $\sigma$, orientation of $\sigma$) orients $Q$, in this case the standard orientation.  If in local coordinates $\lambda$ denotes the straight path from $(-1,0,0,0)$ to $(0,0,0,0)$, then we say that $\gamma$ is obtained from $\alpha$ by $\lambda$-\emph{spinning}.\end{definition}  

\begin{remarks} i) The inverse $\tau^{-1}$ of $\tau$ corresponds to going around $Q$ the other way, thereby reversing the orientation of $Q$ and hence the linking number.

ii) Up to homotopy in $\Emb(L,M;L_0)$, $\lambda$-spinning depends only on the path homotopy class of $\lambda$ and the linking number.  
 \end{remarks}

\setlength{\tabcolsep}{60pt}
\begin{figure}
 \centering
\begin{tabular}{ c c }
 $\includegraphics[width=5in]{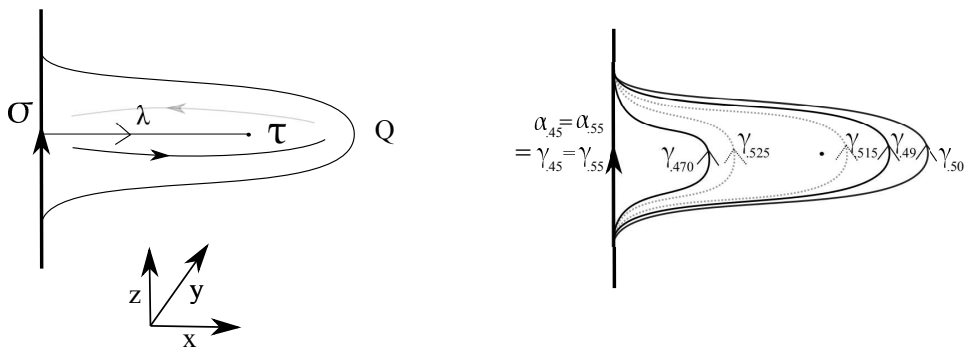}$  
\end{tabular}
 \caption[(a) X; (b) Y]{\label{FigureD,1}
 \begin{tabular}[t]{ @{} r @{\ } l @{}}
Obtaining $\gamma$ by $\lambda$-spinnning $\alpha$ \\
\end{tabular}}
\end{figure}


\begin{notation}  Let $I_0$ be a properly embedded $[0,1]$ in the 4-manifold $M$ and let $1_{I_0}$ denote the identity element in $\pi_1^D(\Emb(I,M;I_0))$.  Let $p<q\in I_0$ and $g\in \pi_1(M,I_0)$ where $I_0$ is viewed as the basepoint, then denote by $\tau_g\in \pi_1^D(\Emb(I,M;I_0))$  the loop obtained by spinning $1_{I_0}$ using a path $\lambda$ from $p$ to $q$ representing $g$.  Let $\tau_{-g}$ denote $\tau_g^{-1}$. \end{notation}

\begin{remarks} i) Spinning can be viewed as the arc pushing map that defines the barbell map of \cite{BG}.  Reversing the orientation of $\lambda$ changes a spin to its inverse up to homotopy in $\Emb(L,M)$.  See Theorem 6.6 \cite{BG}.   Do not confuse $\tau_{-g}=\tau_g^{-1}$ with $\tau_{g^{-1}}$.  

ii)  Modifying the orientation preserving parametrization of $B$, e.g., by an element of $\pi_1(SO(3))$ as one moves along $\lambda$, does not change the path homotopy class of $\gamma$.  See Remark 6.4 i) \cite{BG}.

iii) The homotopy class of $\gamma$  is independent of the representative of $\lambda$.  In particular $\tau_g$ is well defined up to homotopy in $\Emb(I,M;I_0)$ and represents an element of $\pi_1^D(\Emb(I,M;I_0))$.  If $g=1\in \pi_1(M,I_0)$, then $\tau_g = 1_{I_0}\in \pi_1^D(\Emb(I,M;I_0))$. 

\end{remarks}

\begin{lemma} Spinning commutes up to homotopy in $\Emb(I,M; I_0)$.\end{lemma}

\begin{proof}  After an isotopy we can assume that the support of the spins are disjoint.\end{proof}

\begin{theorem} (Dax Isomorphism Theorem)  Let $I_0$ be an oriented properly embedded $[0,1]$ in the oriented 4-manifold $M$.  Then 

i) There is a homomorphism $d_3:\pi_3(M, x_0)\to \BZ[\pi_1(M)\setminus 1]$ with image $D(I_0)$, called the \emph{Dax kernal}.  



ii) $\pi_1^D(Emb(I,M;I_0))$ is generated by $\{\tau_g|g\neq 1, g\in \pi_1(M)\} $ and canonically isomorphic to $\BZ[\pi_1(M)\setminus 1]/D(I_0)$.  \end{theorem}

\noindent\emph{Proof}.   
Let $\alpha=\alpha_t, t\in I$ represent an element of $\pi_1^D(\Emb(I,M;I_0))$.  Being in the Dax group, there exists a homotopy $\alpha_{t,u} \in \Maps(I,M;I_0) $ such that  $ \alpha_{t,u}$ equals $1_{I_0}$ for $u$ near $0$ and $\alpha_{t,u}$ equals $\alpha_t $ for $u$ near 1.  
\vskip 8pt
\noindent\emph{Step 1}:  Define $d(\alpha_{t,u})\in \BZ[\pi_1(M)\setminus 1]$.
\vskip 8pt
As in \cite{Da} define $F_0:I\times I^2 \to M\times I^2$ by $F_0(s,t,u)=(\alpha_{t,u}(s),t, u)$.  As in Chapter III  \cite{Da} we can assume that $F_0$ is \emph{parfait}, in particular is an immersion, has finitely many double points and no triple points.  Furthermore, $F_0$ is self transverse at the double points which we can assume occur at distinct values of the last factor.  The results in Chapter III are stated for closed manifolds but apply to manifolds with boundary since the support of the modification occurs away from the boundary.  See also Chapter VI \cite{Da} which mentions the bounded case. 

\setlength{\tabcolsep}{60pt}
\begin{figure}
 \centering
\begin{tabular}{ c c }
 $\includegraphics[width=5in]{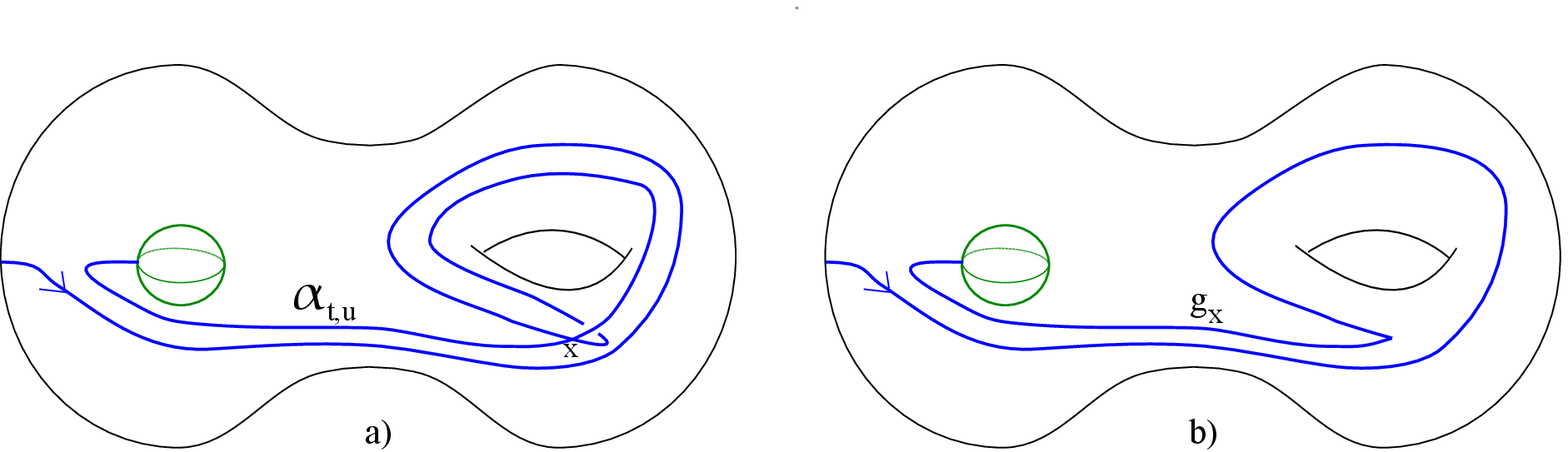}$  
\end{tabular}
 \caption[(a) X; (b) Y]{\label{Figure3,2}
 \begin{tabular}[t]{ @{} r @{\ } l @{}}
\end{tabular}}
\end{figure}

Assign a generator $\sigma_x g_x\in \BZ[\pi_1(M)]$ to each double point $x$ as follows.   Suppose $x=\alpha_{t,u}(p)=\alpha_{t,u}(q)$, where $p<q$.  Let $g_x\in \pi_1(M,I_0)$ be represented by $\alpha_{t,u}|[0,p]*\alpha_{t,u}|[q,1]$.  See Figure \ref{Figure3,2}.  Note that $I_0$ functions as the basepoint.  Let $\sigma_x$ be the self intersection number obtained by comparing the orientation of $DF_0(T_{p,t,u}(I^3))\oplus DF_0(T_{q,t,u}(I^3))$ with that of $T_x(M\times I^2)$.  If $x_1, \cdots, x_n$ are the double points with $g_{x_i}\neq 1$, then define $d(\alpha_{t,u})=\sum_{i=1}^n \sigma_{x_i} g_{x_i}$.
\vskip 8pt

The next two steps show that modulo $D(I_0)$, different choices of $\alpha_{t,u}$ give the same $d$ value. 
\vskip 8pt

\noindent\emph{Step 2}:  If $\alpha^0_{t,u}$ is properly homotopic to $\alpha^1_{t,u}$, then $d(\alpha^0_{t,u})=d(\alpha^1_{t,u})$.
\vskip 8pt
\noindent\emph{Proof}.  By properly homotopic we mean that there exists $\alpha^v_{t,u}, v\in I$ such that each $ \alpha^v_{t,u}\in \Maps(I,M, I_0)$,  $\alpha^v_{t,1}$, $v\in I$ is a homotopy in $\Emb(I,M, I_0)$ from $\alpha^0_{t,1}$ to $\alpha^1_{t,1}$ and $ \alpha^v_{t,u}$ equals $1_{I_0}$ for $u$ near $0$ and $v\in I$.  

Suppose that we have two homotopies $F_0, F_1$ as in Step 1, that are homotopic rel $\partial$.  Then we can interpolate by maps $F_v$ and combine them to a map $F:(I\times I\times I)\times I\to (M\times I\times I)\times I$, such that $F(s,t,u,v)=(\alpha^v_{t,u}(s),t,u,v)$.  Again, we can assume that $F$ is parfait and hence away from finitely many singularities $F$ is a self transverse immersion without triple points.  The double points form a 1-manifold whose endpoints in the interior of $M\times I^3$ occur at singularities.  The local form of a singularity, p.332 \cite{Da}, implies that a double point $x$ sufficiently close to a singular point has $g_x= 1$.  Indeed, since  each $\alpha^v_{t,u}$ is path homotopic to $I_0$, if $x=\alpha^v_{t,u}(r)=\alpha^v_{t,u}(s)$,  then $g_x=1$ when the loop $\alpha^v_{t,u}|[r,s]$ is homotopically trivial.  Here, that loop is homotopically trivial since its diameter converges to $0$ as $x$ approaches the singular point.  Finally, use the other double curves to equate the $d$ values coming from $F_0$ and $F_1$.  \qed

\vskip 8pt  

If $\pi_3(M)\neq 0$, then there will be non homotopic null homotopies of $\alpha_t$ in $\Maps(I,M; I_0)$ which may lead to different values of $d(\alpha_{t,u})$.  The Dax kernal keeps track of this indetermanency.  Call an $\alpha_{t,u}$ a \emph{kernal map} if for all $u$ close to either 0 or 1, $\alpha_{t,u}=1_{I_0}$.  In a natural way, up to homotopy supported away from $\partial I^3$ there is a natural isomorphism between kernal maps and $\pi_3(M, x_0)$, where $x_0=I_0(1/2)$ and the addition of kernal maps is given by concatenation.

\begin{definition}   Define $d_3:\pi_3(M,x_0)\to \BZ[\pi_1(M)\setminus 1]$ as follows.  Represent $a\in \pi_3(M,x_0)$ as a kernal map $\alpha_{t,u}$.  Now define $d(a)=d(\alpha_{t,u})\in \BZ[\pi_1(M)\setminus 1]$ as in Step 1.  Define $D(I_0)=d_3(\pi_3(M,x_0))$.  When $I_0$ is clear from context, we will write $D(I_0)$ as $D$.  \end{definition}

\noindent\emph{Step 3}: $d_3:\pi_3(M)\to \BZ[\pi_1(M)\setminus 1]$ is a homomorphism as is $d: \pi_1^D(\Emb(I,M;I_0))\to \BZ[\pi_1(M)\setminus 1]/D$ where $d(\alpha_t):=d(\alpha_{t,u})$ for some $\alpha_{t,u}$.
\vskip 8pt
\noindent\emph{Proof}.  The proof of Step 2 shows that $d_3:\pi_3(M)\to \BZ[\pi_1(M)\setminus 1]$ is well defined.   Its additivity with respect to concatenation shows that it is a homomorphism. If $\alpha^0_{t,u}$, $\alpha^1_{t,u}$ are two null homotopies of  $\alpha_t$ in $\Maps(I,M;I_0)$, then after concatenating with a kernal map we obtain a new null homotopy whose $d$ value differs by an element of $D$.    It follows that $d: \pi_1^D(\Emb(I,M;I_0))\to \BZ[\pi_1(M)\setminus 1]/D$ is well defined.

To show that $d$ is a homomorphism  first observe that $d(1_{I_0})=0$. By concatenating $F_0$'s for $\alpha$ and $\beta$ we see that $d(\alpha*\beta)=d(\alpha)+d(\beta)$.\qed


\vskip 8pt
\noindent\emph{Step 4}:  If $[\alpha]\in \pi_1^D(\Emb(I,M;I_0))$ and without cancellation $d(\alpha_{t,u}) = \sigma_{x_1} g_{x_1} + \cdots + \sigma_{x_n} g_{x_n}$, then $\alpha$ is homotopic to the compositions of spin maps $\tau_{\sigma_{x_1} g_{x_1}}, \cdots, \tau_{\sigma_{x_n} g_{x_n}}$.
\vskip 8pt
\noindent\emph{Proof}.  Let $F_0:I\times I\times I \to M\times I^2$ as in Step 1.  We prove Step 3 by induction on the number of double points.  Assume for the moment Step 3 is true if $F_0$ has $\le k$ double points where $k\ge 1$.  If $F_0$ has  $k+1$ double points, then by changing coordinates we can assume that one occurs at $x=F_0(p,\frac{1}{2},\frac{1}{2})=F_0(q,\frac{1}{2}, \frac{1}{2})$ where $p<q$ and the others occur at $F_0(s,t,u)$ where $u>3/4$.  Thus, $F_0|I\times I\times 5/8$ is homotopic to a spin map 
$\tau$ and there is a homotopy $G_0$ from $1_{I_0}$ to $\tau^{-1}* \alpha$ with $k$ double points of the same group ring types as $F_0|I\times I\times [5/8,1]$ and hence the result follows by induction.  
\vskip 8pt
We now consider the case that there is a single double point.   By modifying the homotopy rel $\partial$ we can assume that with respect to local coordinates on $M\times I\times I$ and  local variables $-\epsilon\le s',t',u' \le \epsilon$; 
\vskip 8pt
$F(q+s',t'+\frac{1}{2},u'+\frac{1}{2})=(0,0,0,-s',t'+\frac{1}{2},u'+\frac{1}{2})$,

$F(p+s',t'+\frac{1}{2},u'+\frac{1}{2})=(u', t', s',0,t'+\frac{1}{2},u'+\frac{1}{2})$ if $\sigma_x=+1$,

$F(p+s',t'+\frac{1}{2},u'+\frac{1}{2})=(u', -t', s',0, t'+\frac{1}{2},u'+\frac{1}{2})$ if $\sigma_x=-1$.  
\vskip 8pt
Thus, the passage from $\alpha_{t,\frac{1}{2}-\epsilon}$ to $\alpha_{t,\frac{1}{2}+\epsilon}$ changes $1_{I_0}$ to $\tau_{ \sigma_x g_x}$, where $g_x$ is the loop $\phi_0*\phi_1$ where $\phi_0$ (resp. $ \phi_1$) is the arc $F_0(p,\frac{1}{2},w), 0\le w \le \frac{1}{2}\ $(resp. $F_0(q,\frac{1}{2},1-w), \frac{1}{2}\le w\le 1)$ which is homotopic to the loop $g_x$. \qed
\vskip 8pt

\noindent\emph{Step 5}:  $d$ is canonical; i.e. if $\alpha$ is a composition of $\tau_{\sigma_1 g_1}, \cdots, \tau_{\sigma_n g_n}$, with all $g_i\neq 1$, then there exists $\alpha_{t,u}$ with $d(\alpha_{t,u})=\sigma_1 g_1 +\cdots +\sigma_n g_n$.
\vskip 8pt
\noindent\emph{Proof}. The local functions defined in Step 4 show how to construct a homotopy $F_0$  from $1_{I_0}$ to $\alpha$ whose double points evaluate to $\sigma_1 g_1, \cdots, \sigma_n g_n.$  \qed
\vskip 8pt

\noindent\emph{Step 6}: $d: \pi_1^D(\Emb(I,M;I_0))\to \BZ[\pi_1(M)\setminus 1]/D$ is an isomorphism.
\vskip 8pt
\noindent\emph{Proof}.  Step 3 and 5 show that $d$ is a surjective homomorphism. We now prove injectivity.  If $\alpha\in \pi_1^D(\Emb(I,M;I_0))$ and $d(\alpha_{u,t})\in \mD$ then by concatenating with a kernal map we can assume that $d(\alpha_{u,t})=0$.  It follows from Step 4 that $\alpha$ is homotopic to a composit of spin maps $\tau_{\sigma_{x_1} g_{x_1}}, \cdots, \tau_{\sigma_{x_n} g_{x_n}}$ whose sum is equal to 0 in $\BZ[\pi_1(M)\setminus 1]$.  Since spin maps commute it follows that $\alpha$ is homotopic to 
$1_{I_0}$.   This completes the proof of the Dax isomorphism theorem. \qed
\vskip 8pt





\begin{theorem} Let M be a 4-manifold such that  $\pi_3(M)=0$, then $\pi_1^D(Emb(I,M;I_0))$ is freely generated by $\{\tau_g|g\neq 1, g\in \pi_1(M)\} $ and canonically isomorphic to $\BZ[\pi_1(M)\setminus 1]$.\qed\end{theorem}

\begin{theorem} If $M=S^1\times B^3\natural S^2\times D^2$, then $\pi_1^D(Emb(I,M;I_0))$ is isomorphic to $\BZ[\BZ\setminus 1]$ and is freely generated by $\{\tau_g|g\neq 1, g\in \pi_1(M)\} $. (Here $\pi_1(M)$ is expressed multiplicatively.) \end{theorem}

\begin{proof}  $\pi_3(M)$ as a $\BZ[\pi_1]$ module is generated by the Hopf map of $S^3$ to a 2-sphere $ Q$ and Whitehead products of conjugates of $\pi_2(Q)$.  Once given $ I_0, Q$ can be chosen disjoint from $I_0$ and hence any element of $\pi_3(M)$ has support in a simply connected subcomplex. \end{proof}

\begin{theorem}\label{connect sum} If $M=S^1\times B^3\# S^2\times D^2$, then $\pi_1^D(Emb(I,M;I_0))$ is isomorphic to $\BZ[\BN]$ and is freely generated by $\{\tau_g|g\ge 1\}$.\end{theorem}

\begin{proof} Here the Dax kernal $\neq 0$.  The various $\pi_1(M)$ conjugates in $\pi_3(M)$ of  the separating $S^3$ give, up to sign, the relations $g^i=g^{-i}$ in $\BZ[\pi_1(M)\setminus 1]$.\end{proof}

\begin{remarks} i) Theorem \ref{dax} is stronger than the one given in \cite{Da} in that we identified generators of $\pi_1^D(\Emb(I,M;I_0))$.  Working with these commuting elements enables us to avoid a parametrized double point elimination argument and the need to modify $F_0$ to eliminate double points $x$ with $g_x=1$.  Also, we have a natural isomorphism of $\pi_1^D(\Emb(I,M;I_0))$ with a computable quotient of the group ring as opposed to one arising from an abstract bundle cobordism construction.

ii) The ordering of $I_0$ enables us to unambiguously define $\sigma_x$ and $g_x$.

iii) We note that the Dax group $\pi_1^D(\Emb(S^1,M; S^1_0))$, has an extra relation from being able to cancel double points of $F_0$ by going around the $S^1$.  Dax computed the case $M=S^1\times S^3$, P. 369 \cite{Da}.  See also \cite{AS} and \cite{BG} for the case $M=S^1\times S^3$.  \end{remarks}

\begin{question} What is the relation between the Dax kernal and the six dimensional self intersection invariant?  \end{question}

\begin{remark} Schneiderman and Teichner \cite{ST} show that for an oriented six dimensional manifold $P$ the self intersection invariant $\mu_3:\pi_1(P)\to \BZ[\pi_1(P)]/<g+g^{-1},1>$ specializes to a map $\mu_3:\pi_3(N)\to \mF_2 T_N$, when $P=N\times I$ and where $T_N$ is the vector space with basis the non trivial torsion elements of $\pi_1(N)$ and $\mF_2$ is the field with two elements.  Our setting is both similar and different in that we are looking at an \emph{ordered} self intersection of mapped \emph{3-balls with fixed boundary} into $M\times I\times I$.  As indicated in Theorem \ref{connect sum} the Dax kernal can be nontrivial, e.g. in manifolds with $\pi_1(M)=\BZ$.\end{remark}

\begin{remarks}  i) Syunji Moriya \cite{Mo} shows that for certain simply connected 4-manifolds M, $\pi_1(\Emb(S^1, M))\cong H_2(M,\BZ)$.

ii) See Danica Kosanovic's  thesis \cite{Ko1} and paper  \cite{Ko2} for results on $\Emb(I,M)$ for general manifolds $M$. \end{remarks}

\section{From Discs to Paths}

\begin{definition}  Let $D_0$ be a properly embedded disc in $M$ with  dual sphere $G$.  Let $\mD$ be the set of isotopy classes $\rel \partial$ of discs homotopic $\rel \partial$ to $D_0$.  If $D_1, D_2\in \mD$, then define $D_1+D_2=D_3$ so that $D_3$ is the realization of a tubed surface whose sector form data is the concatenation of that of $D_1$ and $D_2$.  This means that if $D_1$ (resp. $D_2)$ has $n_1$ (resp. $n_2$) sectors with data then $D_3$ has $n_1+n_2$ sectors with the corresponding data. \end{definition}

\begin{proposition}  $\mD$ is an abelian group with unit $[D_0]$ under the operation $+$.\end{proposition} 

\begin{proof} We need to show that $D_3$ is independent of the choice of representatives of $D_1$ and $D_2$, the other conditions being immediate.  In particular, by Lemma \ref{normal permutation} $D_3$ is independent of the concatenation order and hence $\mD$ is abelian.  We can assume that $D_1$ coincides with $D_0$ near their boundaries, so an isotopy of $D_1$ to $D_1'$ can be chosen to be supported away from some neighborhood of $\partial D_0$.  Since the data of $D_2$, except for its framed embedded paths, can be isotoped within their sectors to be very close to $\partial D_0$,  we see that the isotopy of $D_1$ can be chosen to avoid it.  While the framed embedded paths associated to $D_2$ may get moved during the ambient isotopy of $D_1$ to $D_1'$, the light bulb lemma enables them to isotope back to their original positions without introducing intersections with $D_1'$. \end{proof}

\begin{remark} $\mD$ is a torsor, where $\BZ[\pi_1(M)\setminus 1]$ and $\BZ[T_2]$ act on $\mD$.  Here $T_2$ is the set of non trivial 2-torsion elements.  The former acts by attaching the appropriate self-referential discs and the latter by attaching the appropriate double tubes.  \end{remark} 

\begin{notation}  If $\lambda$ is a framed embedded path with endpoints in $D_0$ representing a nontrivial 2-torsion element of $\pi_1(M)$, then let $\hat\lambda$ denote this element and  let $D_\lambda$ denote the realization of the self-referential form tubed surface whose data consists exactly of $(\lambda)$. If $1\neq g\in \pi_1(M)$, then let $D_g$ (resp. $D_{-g})$ denote the realization of the self-referential form tubed surface whose data only consists exactly of $(+g)$ (resp. $(-g)$.\end{notation}

\begin{remark}  Since an element of $\mD$ can be put into self-referential form it follows that the $D_g$'s and $D_\lambda$'s are generators of $\mD$.\end{remark}

\begin{definition} \label{novel} Let  $D_0$ be a properly embedded disc in the 4-manifold $M$, not necessarily with a dual sphere.  View $D_0$ as $I\times I$ with $I_0$ denoting $I\times 1/2$ and $ \mF_0$  this product foliation. If $D$ is another properly embedded disc that agrees with $D_0$ along $\partial D_0$, then $D$ gives rise to an  element $[\phi_{D_0}(D)]\in \pi_1(\Emb(I,M;I_0))$, where $\Emb(I,M;I_0)$ is the space of smooth embeddings of $I$ based at $I_0$.     To construct $\phi_{D_0}(D)$, first isotope $D$ to coincide with $D_0$ near $\partial D_0$.  Next view $ D=I\times I$ where this foliation $\mF$ coincides with $\mF_0$ near  $\partial D_0$.   Use $D_0$ to inform how to modify $\mF$ to a  loop $\phi_{D_0}(D)$ in $\Emb(I,M;I_0)$ based at $ I_0$. To do this first define $\beta\in \Emb(I,M)$ as follows. For $t\in [0,1/4], \beta_t$ traces $I\times(1/2-2t)$ using $\mF_0$; for $t\in [1/4,3/4], \beta_t$ traces $I\times(2t-.5)$ using $\mF$; and for $t\in [3/4,1], \beta_t$ traces $I\times(1.5-2t)$ using $\mF_0$.  Naturally modify the ends of each $\beta_t$ to coincide with $I_0$ near $ \beta_t(0)$ and $\beta_t(1)$ to obtain $\phi_{D_0}(D)$ with $[\phi_{D_0}(D)]$ denoting the corresponding class in $\pi_1(\Emb(I,M;I_0))$.\end{definition}

\begin{remark} For the sake of exposition, $D_0$ was parametrized as a disc with corners.  The definition is readily modified to the smooth setting.\end{remark}

Since $\Diff(D^2\fix \partial)$ is connected \cite{Sm3} it follows that $\phi_{D_0}$ is well defined and depends only on $D_0$ and $I_0$.  If $\mD$ is the set of isotopy classes of discs homotopic to $D_0 \rel \partial$, then together with the Dax isomorphism theorem we obtain the following result.  

\begin{theorem} Let $D_0$ be a  properly embedded disc in the oriented 4-manifold, $I_0$ an oriented properly embedded arc in $D_0$  and $\mD$ be the isotopy classes of embedded discs homotopic $\rel \partial$ to $D_0$, then there is a canonical function $\phi_{D_0}: \mD \to \BZ[\pi_1(M)\setminus 1]/D(I_0)$ such that if $D$ is a embedded disc homotopic rel $\partial$ to $D_0$, then $\phi_{D_0}([D])\neq 0$ implies $D$ is not isotopic to $D_0\rel\partial$.\end{theorem}

We have more algebraic structure when $D_0$ has a dual sphere.  The following is a sharper form of Theorem \ref{main} ii) of the introduction.



\begin{theorem}  \label{main ii} Let  $D_0\subset M$ be a properly embedded disc with the  dual sphere $G$ and $\mD$ the isotopy classes of discs homotopic to $D_0\rel \partial D_0$.  Then $\mD$ is an abelian group with zero element $[D_0]$ and there exists a natural homomorphism $\phi_{D_0}:\mD\to \BZ[\pi_1(M)\setminus 1]/D(I_0)\cong \pi_1^D(\Emb(I,M; I_0))$, where $D(I_0)$ is the Dax kernal, such that the generators of $\mD$ are mapped as follows. 

i) $\phi_{D_0}([D_\lambda])=\hat \lambda$

ii) $\phi_{D_0}([D_g])=g+g^{-1}$.\end{theorem}

\setlength{\tabcolsep}{60pt}
\begin{figure}
 \centering
\begin{tabular}{ c c }
 $\includegraphics[width=5in]{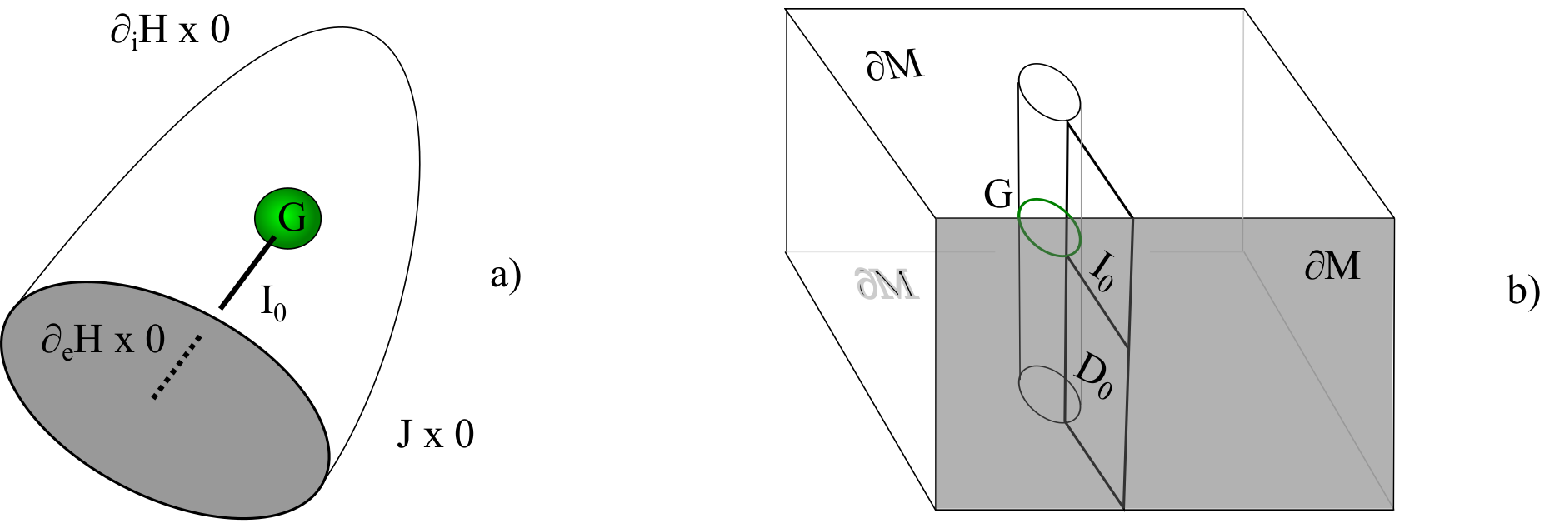}$  
\end{tabular}
 \caption[(a) X; (b) Y]{\label{Figure4,1}
 \begin{tabular}[t]{ @{} r @{\ } l @{}}
\end{tabular}}
\end{figure}

\begin{proof}  We first set the local picture.  View $N(D_0\cup G)$ as the manifold with corners $J\times [-1,1]$, where $J=H\setminus \inte(B)$, where $B$ is an open 3-ball and $H$ is a half 3-ball with $\partial H = \partial_e H\cup \partial_i H$, the \emph{external} and \emph{internal boundaries}.  Also, $\partial M\cap J\times [-1,1]=(\partial_e H\cup \partial B)\times [-1,1]\cup J\times \{-1,1\}$.  Here $G_t:=\partial B\times t$ and $N(G)\cap \partial M=G\times [-1,1]$.  $D_0$ is a vertical disc in $J\times [-1,1]$ with $I_t:=D_0\cap J\times t$,  where $I_0$ is an arc from $\partial_e H\times 0$ to $G:=G_0$.  See Figure \ref{Figure4,1} a).  Figure \ref{Figure4,1} b) shows a one dimension lower version.  In that figure $G$ is a circle and $D_0$ is a disc.  $\partial M$ is the union of $G\times [-1,1]$ and the shaded face which is the analogue of $\partial_e(H)\times [-1,1]$ and the top and bottom faces.

We now define  $\phi_{D_0}$ from this point of view.  If $D$ is a properly embedded disc that coincides with $D_0$ near $\partial D$, then the $I_t$ fibering of $D_0$ induces $\phi_{D_0}(D)\in\pi_1^D(\Emb(I,M;I_0))$ as follows.  It first induces a map $\phi_{D_0}':[-1,1]\to (\Maps: [-1,1]\to \Emb(I,M))$.  The projection of $I_t$ to $I_0$  then informs how to close up to a loop and modify the ends to coincide with $I_0$ to obtain a well defined element of $\pi_1^D(\Emb(I,M;I_0))$.  It is a homomorphism since by construction $\phi_{D_0}([D_0])]=[1_{I_0}]$.  Since addition is given by concatenation of sector forms it follows that $\phi_{D_0}([D_1]+[D_2])=\phi_{D_0}([D_1])+\phi_{D_0}([D_2])$.

\setlength{\tabcolsep}{60pt}
\begin{figure}
 \centering
\begin{tabular}{ c c }
 $\includegraphics[width=1.5in]{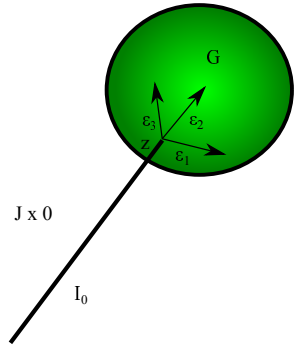}$  
\end{tabular}
 \caption[(a) X; (b) Y]{\label{Figure4,2}
 \begin{tabular}[t]{ @{} r @{\ } l @{}}
Orientations on $D_0$ and $G$
\end{tabular}}
\end{figure}

\setlength{\tabcolsep}{60pt}
\begin{figure}
 \centering
\begin{tabular}{ c c }
 $\includegraphics[width=5in]{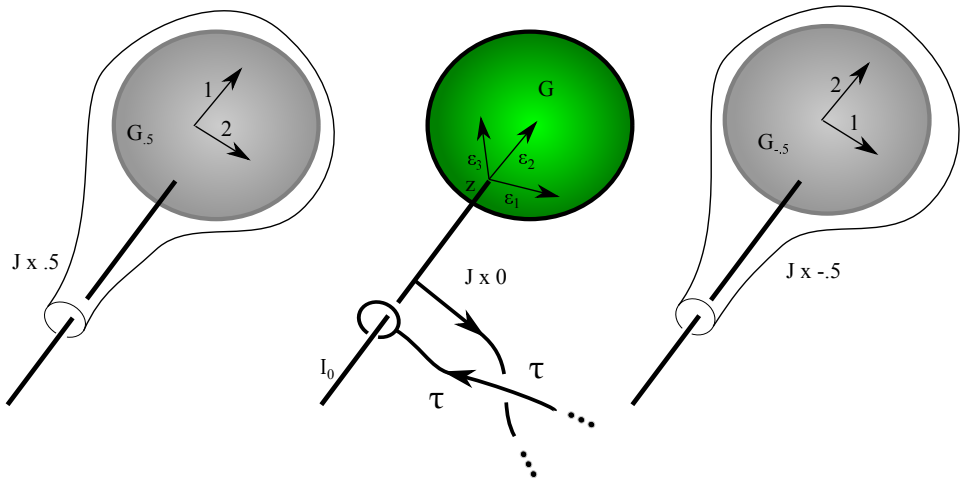}$  
\end{tabular}
 \caption[(a) X; (b) Y]{\label{Figure4,3}
 \begin{tabular}[t]{ @{} r @{\ } l @{}}
Orientation on $P(\alpha)$
\end{tabular}}
\end{figure}

\setlength{\tabcolsep}{60pt}
\begin{figure}
 \centering
\begin{tabular}{ c c }
 $\includegraphics[width=4in]{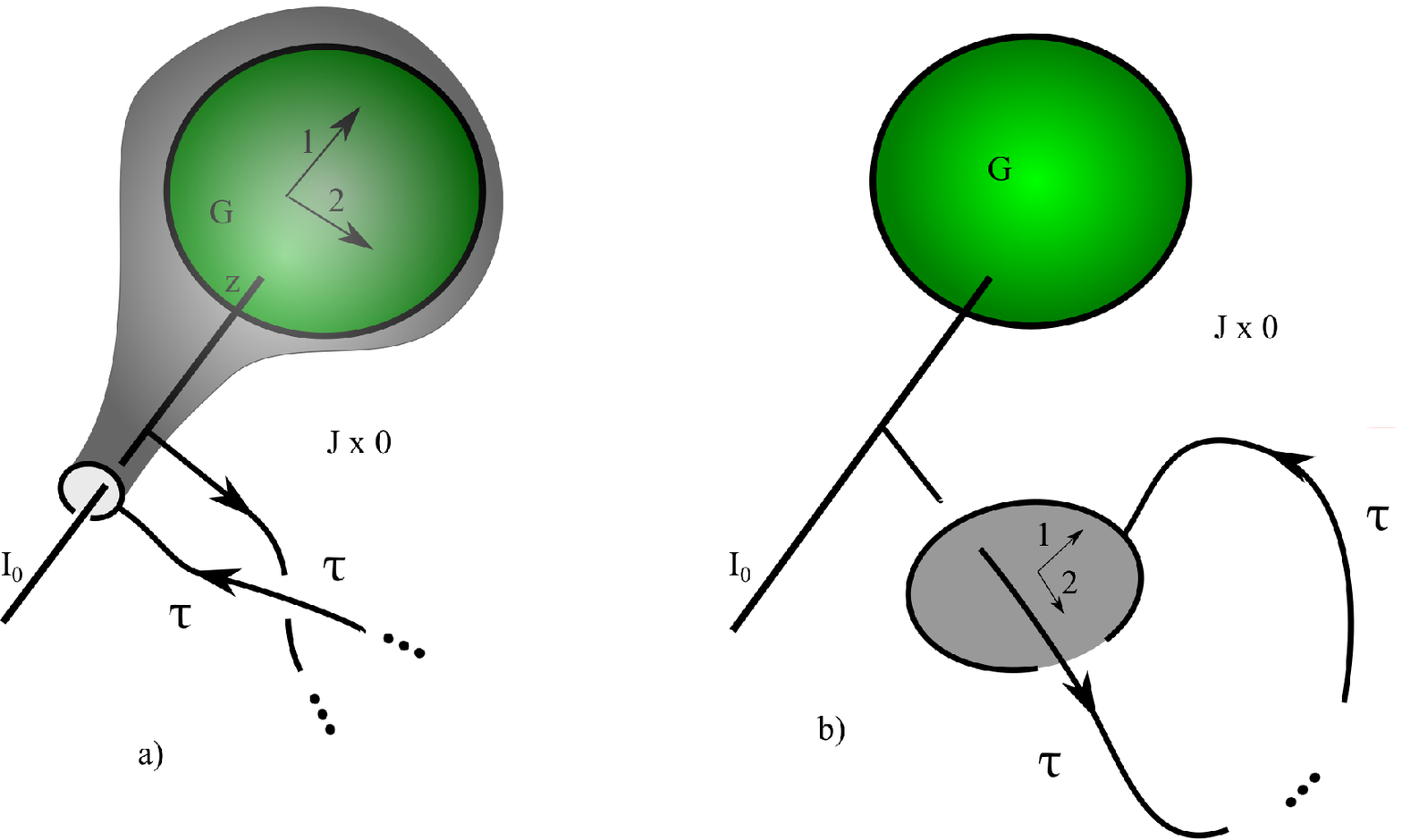}$  
\end{tabular}
 \caption[(a) X; (b) Y]{\label{Figure4,4}
 \begin{tabular}[t]{ @{} r @{\ } l @{}}
Isotoping to a self-referential disc I
\end{tabular}}
\end{figure}

\setlength{\tabcolsep}{60pt}
\begin{figure}
 \centering
\begin{tabular}{ c c }
 $\includegraphics[width=1.75in]{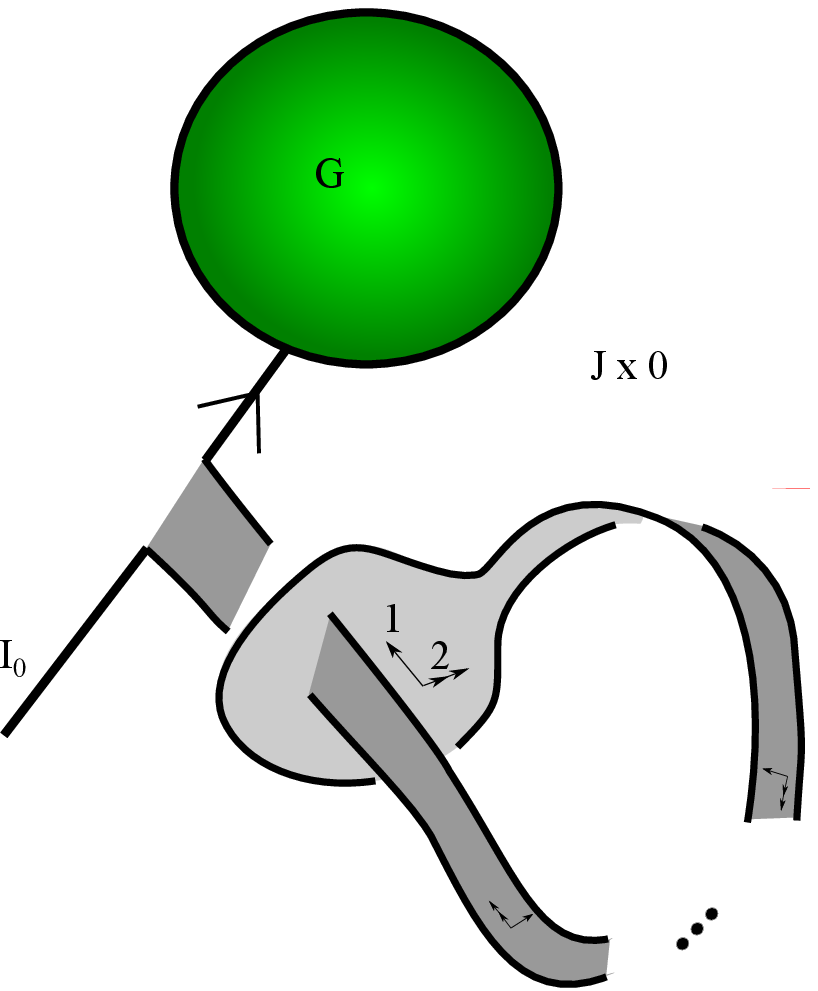}$  
\end{tabular}
 \caption[(a) X; (b) Y]{\label{Figure4,5}
 \begin{tabular}[t]{ @{} r @{\ } l @{}}
Isotoping to a self-referential disc II
\end{tabular}}
\end{figure}

We now show ii).  Given $D_g\in \mD$, represent $\phi_{D_0}(D_g)$ as $\alpha_t$ a loop in $\Emb(I,M;I_0)$.  As in \S 3 we construct a homotopy $\alpha_{t,u}$ in $ \Maps(I,M;I_0)$ from $ \alpha_t $ to $1_{I_0}$ and then compute $d(\alpha_{t,u})$.  To compute the required intersection numbers we need to establish and keep track of  orientations.  First $J\times [-1,1]$ has the standard orientation $(\epsilon_1, \epsilon_2, \epsilon_3, \epsilon_4)$ induced from $\BR^3\times \BR$.  Figure \ref{Figure4,2} shows our orientations on $D_0$ and $G$ as seen from $J\times 0$.  Here $T_z(D_0)$ is oriented by $(\epsilon_2,\epsilon_4)$ and $T_z(G)$ is oriented by $(\epsilon_3,\epsilon_1)$.  Note that $<D_0,G>_z=1$.  Recall that $D_g$ is obtained by coherently tubing $D_0$ with the oriented sphere $P(\alpha)$ along a path $\tau$ representing $g$,  so to know the orientation on $D_g$ it remains to know the orientation of $P(\alpha)$ which is  shown in Figure \ref{Figure4,3}.  The numbers next to the vectors indicate which goes first.  Recall that $P(\alpha)$ is obtained by tubing two copies of $G$, say $G_{ - .5}$ and $G_{.5} $ where the orientation of $G\times -.5$ (resp. $G\times +.5$) disagrees (resp. agrees) with that of $G$.  Figure \ref{Figure4,4} a) shows the projection of $P(\alpha)\cup D_0\cup \tau$ to $J\times 0$; the solid line indicating intersection with the present and shading indicates projection from either the past or future.  Here $J_t; t<0, t=0$ or $t>0$ refers to the past, present or future.  The orientation shown is that of the projection of the disc from the future.  Figure \ref{Figure4,4} b) is another projection after an isotopy of $P(\alpha)\cup \tau$.  To obtain the full picture of this $D_g$ we coherently connect $D_0$ to this isotoped $P(\alpha)$ by the tube $T_\tau$ that follows the isotoped $\tau$.  See Figure \ref{Figure4,5}.

\setlength{\tabcolsep}{60pt}
\begin{figure}
 \centering
\begin{tabular}{ c c }
 $\includegraphics[width=4in]{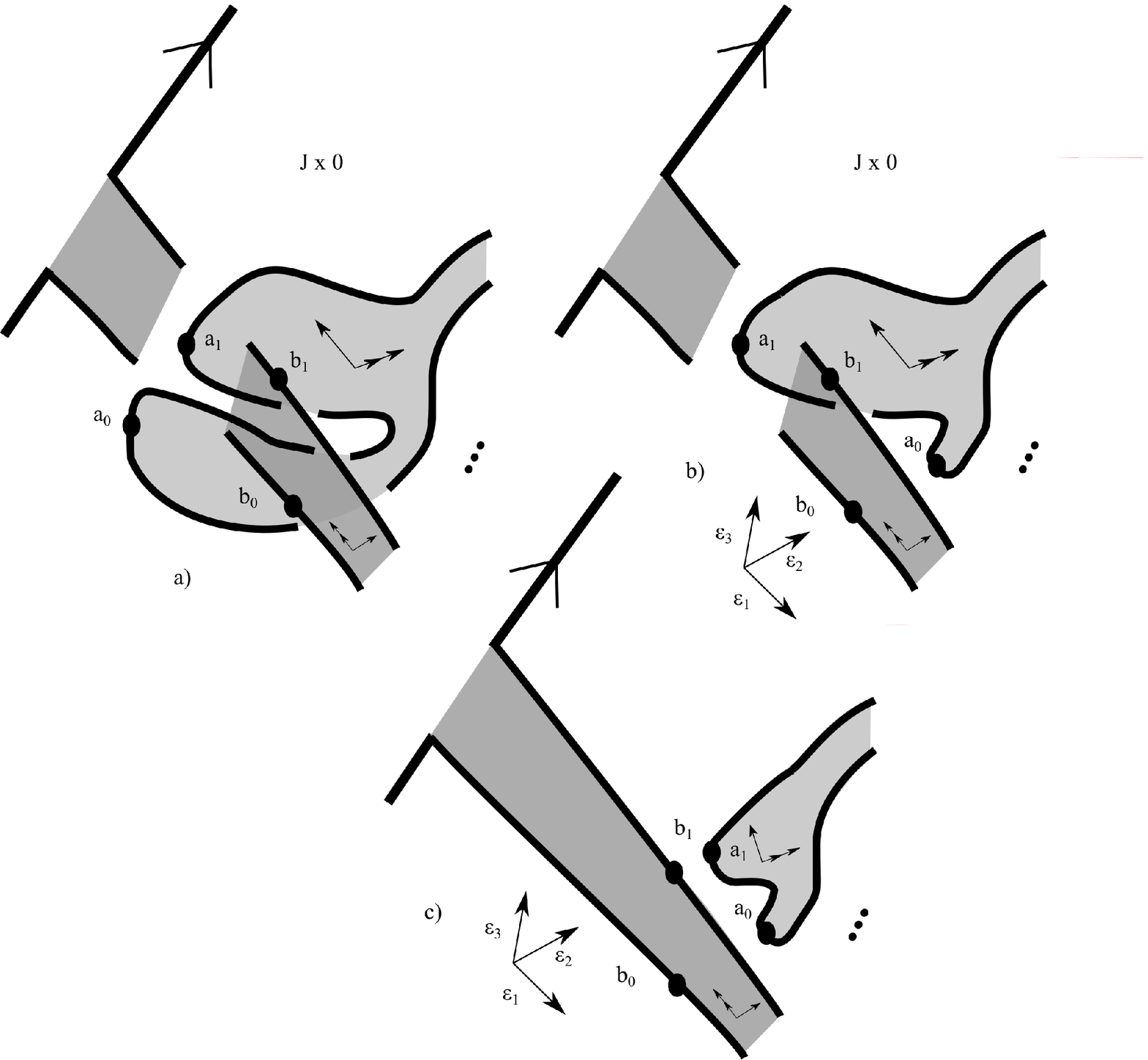}$  
\end{tabular}
 \caption[(a) X; (b) Y]{\label{Figure4,6}
 \begin{tabular}[t]{ @{} r @{\ } l @{}}
 Computing the Intersection Numbers
\end{tabular}}
\end{figure}

We now describe $\alpha_{t,u}$.  The passage from the original $D_g$ to the above one induces a homotopy of $\alpha_{t,0}$ to $\alpha_{t,1/4}$.  Here is a description of the loop $\alpha_{t,1/4}, t\in [-1,1]$.  Starting at $\alpha_{-1,1/4} =I_0$, keeping neighborhood of $\partial I_0$ fixed, $\alpha_{t,1/4} $ sweeps out along $T_\tau$ staying slightly in the past, then remaining slightly in the past  continues across $P(\alpha)$ to reach $\alpha_{1/2,1/4}$, the dark line in Figure \ref{Figure4,5} which is totally in the present.  It then sweeps back across $P(\alpha)$ staying slightly in the future and then back across $T_\tau$ before returning to $I_0=\alpha_{1,1/4}$.  Our homotopy $\alpha_{t,u}$ will have the feature that for all $u, \alpha_{1/2,u}\cap J\times [-1,1]\subset J\times 0$.  If $D_g(u)$ denotes the image of $\alpha_{t,u}, t\in [-1,1]$, then  Figure \ref{Figure4,5} shows the projection of $D_g(1/4)$ to $J\times 0$.  We now homotope $D_g(1/4)$ to $D_g(3/8)$ as shown in Figure \ref{Figure4,6} a).  Here we abuse notation by conflating the domain with the image.  While the embedded part of $D_g(u)$ now becomes  immersed, the homotopy induces a homotopy of $\alpha_{t,1/4}$ to $\alpha_{t, 3/8}$  as loops in $\Emb(I,M; I_0)$.  Figure \ref{Figure4,6} b), (resp. Figure \ref{Figure4,6} c)) shows the result of a further homotopy to $\alpha_{t,9/16}$ (resp. $\alpha_{t, 3/4}$) this time as loops in $\Maps(I,M;I_0)$.  $\alpha_{t,u}$ fails to be a loop in $\Emb(I,M;I_0)$ when $u=1/2$ and $5/8$.  This can be done so that at $u=1/2$ (resp. $u=5/8)$ there is a single self-intersection when $t=1/2$ and $s=a_0$ and $s=b_0$ (resp. $t=1/2$ and $s=a_1$ and $s=b_1$.)  Note that the loop $\alpha_{t,3/4}$ is homotopic in loops $\Emb(I,M;I_0)$ to $1_{I_0}$.  Use this homotopy to complete the construction of $\alpha_{t,u}$.

We now compute the self-intersection values.  Recall that $I_0$ is oriented to point into $G$.  Following the rules of \S 3, since $b_0<a_0$ the group element to this self-intersection is $g^{-1}$.  With notation as in \S 3 we now compute the sign of the self-intersection by comparing    $DF_{0_{b_0, 1/2, 1/2}}(T_{b_0,1/2,1/2}(I^3))\oplus DF_{0_{a_0,1/2, 1/2}}(T_{a_0,1/2,1/2}(I^3))$ with that of $T_{x_1,1/2,1/2}(M\times I^2)$ where $x_1=\alpha(1/2,1/2)(a_0)=\alpha(1/2,1/2)(b_0)$.  Parametrized as in \S 3 we have $DF_{0_{b_0,1/2,1/2}}(\partial/\partial s, \partial/ \partial t, \partial/\partial u)= (\epsilon_1, \epsilon_5, \epsilon_6)$ and $DF_{0_{a_0, 1/2,1/2}} (\partial/\partial s, \partial/\partial t, \partial/\partial u)= (\epsilon_3, \epsilon_4+\epsilon_5, \epsilon_2 + \epsilon_6)$ which as a 6-vector is equivalent to $(\epsilon_1, \epsilon_5,\epsilon_6, \epsilon_3, \epsilon_4, \epsilon_2)$ which is equivalent to the standard basis, hence the self-intersection number is $+1$.  Since $a_1<b_1$, A similar calculation shows that at the second self-intersection the group element is g and the 6-tuple of vectors is $(\epsilon_3, \epsilon_4+\epsilon_5, \epsilon_2+\epsilon_6, -\epsilon_1, \epsilon_5, \epsilon_6)$ which is equivalent to $(\epsilon_3, \epsilon_4, \epsilon_2, -\epsilon_1, \epsilon_5, \epsilon_6)$ which also gives the standard basis.  Therefore, $\phi(D_g)=d(\alpha_{t,u})=g+g^{-1}$.

\setlength{\tabcolsep}{60pt}
\begin{figure}
 \centering
\begin{tabular}{ c c }
 $\includegraphics[width=5in]{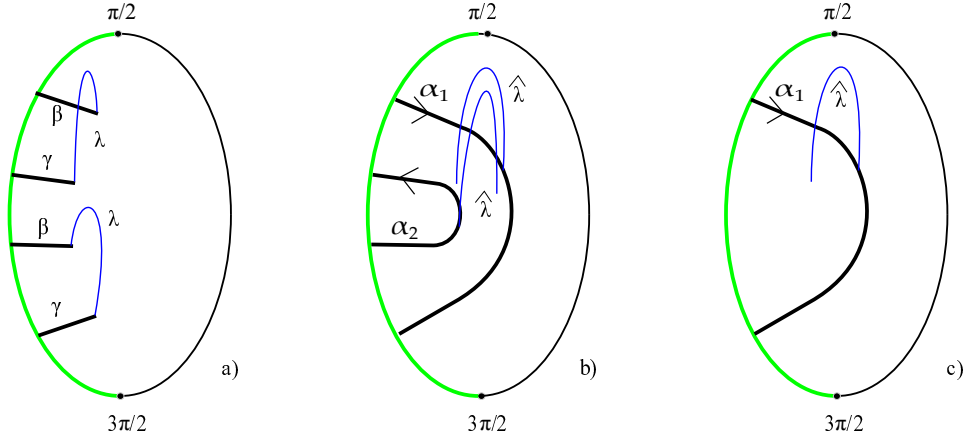}$  
\end{tabular}
 \caption[(a) X; (b) Y]{\label{Figure4,7}
 \begin{tabular}[t]{ @{} r @{\ } l @{}}
Two double tubes equals one single tube
\end{tabular}}
\end{figure}

We now show i) by proving that $2\phi_{D_0}(D_\lambda)=\phi_{D_0}(2D_\lambda) =2\hat\lambda$. Figure \ref{Figure4,7} a) shows a tubed surface with self-referential form data  $(\lambda, \lambda)$.  Figure \ref{Figure4,7} b) shows the result of applying the operation of \S 6 \cite{Ga} to this tubed surface.  Tube sliding moves allow for the $q$ point to $\alpha_2$ to be placed to either side of $\alpha_1$ and vice versa.  Note that the orientations on the $\alpha$ curves are determined by the sign convention.    As in \S 2, deleting the data corresponding to the $\alpha_2$ curve does not change the realization since it's  $q$ point lies on the far side of the approach interval.  What's left is a tubed surface of Figure \ref{Figure4,7} c) with self-referential form data $(+\hat \lambda)$ whose realization is $D_{\hat\lambda}$.  By ii), $\phi_{D_0}(D_{\hat\lambda})= 2\hat \lambda$.\end{proof}

\begin{corollary} \label{knotted disc} Let $M=S^2\times B^2\natural S^1\times B^3, D_0$ be the standard 2-disc as in Figure 2 and $g$ be a generator of $\pi_1(M)$.  Then the discs $D_{g^i}$, $i \in \BN$ are pairwise not  properly isotopic.  On the other hand each $D_{g^i}$ is concordant to $D_0$.\end{corollary}

\begin{proof}  By Theorem \ref{connect sum}, the Dax kernal $D(I_0)=0$. It follows that if $i\neq j$, then $D_{g^i}$ is not isotopic to $D_{g^j}$ since $g^i+g^{-i}\neq g^j+g^{-j}$.  Since each $D_{g^i}$ differs from $D_0$ by a ribbon 3-disc, they are concordant.  See Figure 2 in the introduction.     \end{proof}


\section{Applications and Questions}

As an application we give  examples of knotted 3-balls in 4-manifolds with boundary.  See \cite{BG} and \cite{Wa} for constructions in closed manifolds.  As a prototype we state a result for $M=S^2\times D^2\natural S^1\times B^3$ and  indicate a generalization to other manifolds. 

\begin{theorem}\label{knotted ball two}  If $M=S^2\times D^2\natural S^1\times B^3$ and $\Delta _0= x_0\times B^3$ in the $S^1\times B^3$ factor, then there exist infinitely many 3-balls  properly homotopic to $\Delta_0$, but not pairwise properly isotopic.\end{theorem}

\begin{remark}  The following result is a straight forward extension of Hannah Schwartz' Lemma 2.3 \cite{Sch} for spheres with dual spheres to discs with dual spheres, with a somewhat different proof. \end{remark}

\begin{lemma} \label{schwartz}  Let $D_0\subset N$ be a properly embedded 2-disc with dual sphere $G$.  If $D_1$ is a properly embedded 2-disc that coincides with $D_0$ near $\partial D_0$ and $D_1$ is homotopic $\rel \partial$ to $D_0$, then there exists a diffeomorphism $\psi:(N,D_0)\to (N, D_1)$.  If $D_1$ is homotopic $\rel \partial$ to $D_0$, then $\psi$ can be chosen to fix a neighborhood of $\partial N$ pointwise.  If $D_0$ is concordant to $D_1$, then $\psi$ can also be chosen to be homotopic to $\id \rel \partial$.\end{lemma}

\begin{proof}  Let $G\times [-\epsilon, \epsilon]$ be a product neighborhood of $G\subset \partial N$.  Let $N_1=N\cup_{G\times [-\epsilon,\epsilon]} B^3\times [-\epsilon,\epsilon]$.  Then $N$ is obtained from $N_1$ by removing a neighborhood of the arc $\kappa=0\times [-\epsilon,\epsilon]$.  Any loop $\gamma \in \Emb(I,N_1;\kappa)$ whose time 1 map preserves the framing of $T(\kappa)$ induces $\psi_1:(N_1,\kappa)\to (N_1,\kappa)$ fixing $\partial N_1\cup N(\kappa)$ pointwise and hence a map $ \psi_\gamma:N\to N$ fixing $\partial N$ pointwise, otherwise it induces a diffeomorphism that twists the boundary.  Such a diffeomorphism is called an \emph{arc pushing map}.

Since $D_0, D_1$ coincide near $N(\partial D_0)$, we can extend slightly to discs $E_1, E_0$ in $N_1$, which coincide in $N_1\setminus N$ with $\partial E_0\subset \kappa\cup \partial N_1$.  Let $\gamma$ be the arc pushing map which  first deformation retracts $E_0$ to a small neighborhood of $\partial E_0$ and then expands along $E_1$.  If $D_1$ is homotopic to $D_0$ such an isotopy can be constructed to preserve the normal framing  of $\kappa$ and hence induce a diffeomorphism $\psi_\gamma:(N,D_0)\to (N,D_1)$ which fixies $N(\partial N)$ pointwise.  

If $\hat\psi_\gamma:N_1\times I\to N_1\times I$ is the map induced from suspending the ambient isotopy induced from $ \gamma$, then $ \kappa$ tracks out a properly embedded disc.  If $D_1$ is concordant to $D_0$, then this disc is isotopic $\rel \partial$ to $\kappa\times I$, in which case $\psi_\gamma$ is homotopic to $\id \rel \partial$.\end{proof}

\begin{remark}  It suffices that $D_1$ and $D_0$ induce the same framing on their boundaries to enable $\psi$ to fix $\partial N$ pointwise.\end{remark}

\noindent\emph{Proof of Theorem \ref{knotted ball two}}.  Let $g$ be a generator of $\pi_1(M)$.  Let $D_i$ be the disc $D_{g^i}$ of Theorem \ref{knotted disc}.  By that result all these $D_i$'s are homotopic, in fact concordant, yet pairwise not isotopic $\rel \partial$.  Apply the lemma to obtain $\psi_i:M\to M$ a diffeomorphism, properly homotopic to $\id$ fixing $N(\partial M)$ pointwise such that $\psi_i(D_0)=D_i$.  Let $\Delta_i=\psi_i(\Delta_0)$.  Since $\Delta_0\cap D_0=\emptyset$ it follows that for all $i$, $\Delta_i\cap D_i=\emptyset$.  If $\Delta_i$ is properly isotopic to $\Delta_j, i\neq j$, then the corresponding ambient isotopy takes $D_i$ to $D'_i$ with $D'_i\cap \Delta_j=\emptyset$.  Now $M\setminus \inte(N(\Delta_0))$ is diffeomorphic to $S^2\times D^2$ and hence so is $M\setminus \inte(N(\Delta_j))$.  Since $\Delta_i'$ is properly homotopic to $\Delta_j$ in $M$. $D_i'$ is homotopic $\rel \partial$ to $D_j$ in this $S^2\times D^2$.  By Theorem 10.4 \cite{Ga}, $D'_i $ is isotopic $\rel \partial$ to $ D_j$, which is a contradiction.\qed

\begin{remark}  In a somewhat similar manner we obtain knotted 3-balls in some manifolds of the form  $W=M\natural S^1\times B^3$ where $D_0\subset M$ has a dual sphere $G\subset M$.  Here $\pi_1(W)=\pi_1(M)*\BZ$.  Let $t$ denote a generator of $\BZ$.  We require that the subgroup of $\BZ[\pi_1(W)\setminus 1]$ generated by $t^n+t^{-n}, n\in \BN$ is not contained in the subgroup generated by $\BZ[\pi_1(M)]+D(I_0)$.  For example, manifolds $W$, where $M$ of the form $S^2\times D^2\natural Y$ and $\pi_3(Y)=0$. Define $\Delta_0=x_0\times B^3$ and let $D_1$ be obtained by attaching self-referential discs to $D_0$ so that $\phi_{D_0}(D_1)\notin \BZ[\pi_1(M)]+D(I_0)$.  Now modify $\Delta_0$ to $\Delta_1$ by embedded surgery so that $\Delta_1\cap D_1=\emptyset$ and $\Delta_1$ is homotopic $\rel \partial$ to $\Delta_0$.  If $\Delta_1$ can be isotoped to $\Delta_0$, then $D_1$ can be isotoped into $M$.  Since $D_1$ is homotopic to $D_0$ in $W$, a homotopy can be constructed to be supported in $M$.  This can be seen be remembering that $\pi_2(W)=H_2(\tilde W)$, so a 2-sphere in $\tilde W$ homologically trivial in $\tilde W$ is homologically trivial in $\tilde W \setminus \piinv(\Delta_0)$, where $\pi$ is the covering projection.  It follows that $\phi_{D_0}(D_1)\in \BZ[\pi_1(M)]+D(I_0)$ a contradiction.

Note that the analogous construction does not work for $V=S^2\times D^2\# S^1\times B^3$ for the standard $D_0$ which lies in the $S^2\times D^2$ factor, since for this $D_0$ homotopy implies isotopy.   That is because the separating 3-sphere can be used to disentangle a single self-referential disc. Also multiple self-referential discs can be disentangled using the separating 3-sphere and the light bulb lemma.     \end{remark}
\vskip 10pt
We conclude with a problem and two questions.

\begin{problem} Complete the isotopy classification of properly embedded discs in 4-manifolds with dual spheres.  \end{problem}

The following question specializes this problem to 4-manifolds without 2-torsion in their fundamental groups?

\begin{questions} Let $D_0\subset M$ be a properly embedded disc with dual sphere $G$ such that $\pi_1(M)$ has no 2-torsion.  Let $\mD$ be the isotopy classes of embedded discs homotopic to $D \rel \partial$.  Let $\phi_{D_0}:\mD\to \BZ[\pi_1(M, z)\setminus 1]/D\cong\Emb(I,M;I_0)$ be the canonical homomorphism.  What is $\ker \phi_{D_0}$?  In particular, if $M=S^2\times D^2\natural S^1\times B^3$, is $D_g$ isotopic $\rel \partial $ to $D_{g^{-1}}$?  \end{questions}


\enddocument